\documentclass[12pt]{amsart}
\usepackage{amsmath,amssymb,amsthm,graphicx,url}
\usepackage{epstopdf}
\usepackage{color}

\newcommand{\be}{\begin{equation}}
\newcommand{\ee}{\end{equation}}

\newcommand{\CI}{{\mathcal C}^\infty }
\newcommand{\CIc}{{\mathcal C}^\infty_{\rm{c}} }

\newcommand{\cH}{{\mathcal H}}

\newcommand{\CC}{{\mathbb C}}
\newcommand{\NN}{{\mathbb N}}
\newcommand{\Q}{{\mathbb Q}}
\newcommand{\RR}{{\mathbb R}}

\newcommand{\TT}{{\mathbb T}}

\newcommand{\ZZ}{{\mathbb Z}}

\newcommand{\vol}{\operatorname{vol}}

\newcommand{\supp}{\operatorname{supp}}

\renewcommand{\Re}{\mathop{\rm Re}\nolimits}
\renewcommand{\Im}{\mathop{\rm Im}\nolimits}

\theoremstyle{plain}

\newtheorem{thm}{Theorem}
\newtheorem{prop}{Proposition}[section]
\newtheorem{cor}[prop]{Corollary}
\newtheorem{lem}[prop]{Lemma}

\theoremstyle{definition}

\numberwithin{equation}{section}

\def\bbbone{{\mathchoice {1\mskip-4mu {\rm{l}}} {1\mskip-4mu {\rm{l}}}
{ 1\mskip-4.5mu {\rm{l}}} { 1\mskip-5mu {\rm{l}}}}}

\def\squarebox#1{\hbox to #1{\hfill\vbox to #1{\vfill}}}

\newcommand{\Id}{{I}}

\usepackage{amsxtra}

\ifx\pdfoutput\undefined
  \DeclareGraphicsExtensions{.pstex, .eps}
\else
  \ifx\pdfoutput\relax
    \DeclareGraphicsExtensions{.pstex, .eps}
  \else
    \ifnum\pdfoutput>0
      \DeclareGraphicsExtensions{.pdf}
    \else
      \DeclareGraphicsExtensions{.pstex, .eps}
    \fi
  \fi
\fi

\title[Control for Schr\"odinger operators on 2-tori]
{Control for Schr\"odinger operators on 2-tori: rough potentials}
\author[J. Bourgain]{Jean Bourgain}
\address{School of Mathematics, Institute for Advanced Study,  Princeton, NJ 08540, USA}
\email{bourgain@math.ias.edu}

\author[N. Burq]{Nicolas Burq}
\address{Universit{\'e} Paris Sud,
Math{\'e}matiques,
B{\^a}t 425, 91405
Orsay Cedex, France,
and Ecole Normale Sup\'erieure, 45, rue d'Ulm, 75005
Paris,  Cedex 05,
France}
\email{Nicolas.burq@math.u-psud.fr}
\author[M. Zworski]{Maciej Zworski}
\address{Mathematics Department, University of California, Berkeley, \
CA 94720, USA}
\email{zworski@math.berkeley.edu}

\thanks{This research was partially supported by NSF grants DMS-0808042, DMS-0835373 (JB) and DMS-1201417 (MZ)}

\usepackage{amssymb}
\usepackage{amsmath, amsthm, amsopn, amsfonts}

\def\11{{\rm 1~\hspace{-1.4ex}l} }
\def\R{\mathbb R}
\def\C{\mathbb C}
\def\Z{\mathbb Z}
\def\N{\mathbb N}

\def\T{\mathbb T}

\begin{document}

\begin{abstract}
For the Schr\"odinger equation, $ ( i \partial_t + \Delta ) u = 0 $ on a torus, an arbitrary non-empty open 
set $ \Omega $ provides control and observability of the solution: $ \| u |_{ t = 0 }
\|_{ L^2 ( \T^2 ) } \leq K_T \| u \|_{ L^2 ( [0,T] \times \Omega )}  $.
We show that the same result remains true for $ ( i \partial_t +
\Delta - V ) u = 0 $ where $ V \in L^2 ( \T^2 ) $, and $ \T^2 $ is a
(rational or irrational) torus. That extends the
results of \cite{AM},  and \cite{BZ4} where the observability was proved for $ V \in C ( \T^2)
$ and conjectured for $ V \in L^\infty ( \T^2 ) $.  The higher
dimensional generalization remains open for $ V \in L^\infty ( \T^n ) $.
\end{abstract}   

\maketitle   

\section{Introduction}   
\label{in}

The purpose of this paper is to prove a case of the conjecture made by
the last two authors in \cite{BZ4}. It concerned control and 
observability for Schr\"odinger operators on tori with $ L^\infty $
potentials. Here we prove that for two dimensional tori the
desired results are valid for potentials which are merely in $ L^2 $.

To state the result consider
\[  \T^2 : = \RR^2 / A \ZZ \times B\ZZ \,, \ \ A, B  \in \RR \setminus \{
 0 \}\,,    \ \  V \in L^2 ( \T^2 ) , \]
\begin{equation}
\label{eq:st}   ( - \Delta + V ( z ) - \lambda ) u ( z) = f ( z ) \,,
\ \    z \in \T^2 , 
\end{equation}
and 
\begin{equation}
\label{eq:dy}
 i \partial_t u ( t, z ) = ( - \Delta + V( z ) ) u ( t , z) \,, \
\ z \in \T^2 \,, \ \ \ 
\end{equation}

The first theorem concerns solutions of the stationary 
Schr\"odinger equation and is applicable to high
energy eigenfunctions:
\begin{thm}
\label{t:1}
Let $ \Omega \subset \T^2 $ be a non-empty open set.
There exists a constant $ K = K ( \Omega ) $, depending only 
on $ \Omega $, such that 
for any solution of \eqref{eq:st} we have
\begin{equation}
\label{eq:t1}
\| u \|_{ L^2 ( \T^2) } \leq K \left( \| f \|_{ L^2 ( \T^2) }  + \| u \|_{ L^2 ( \Omega ) } \right) \,. 
\end{equation}
\end{thm}

Theorem \ref{t:1} can be deduced from the following dynamical result:
\begin{thm}
\label{th.2}
Let $ \Omega \subset \T^2$ be a non empty open set and let $ T > 0 $.
There exists a constant $ K  $, depending only 
on $ \Omega $, $ T $ and $ V $, such that 
for any solution of \eqref{eq:dy} we have
\begin{equation}
\label{eq:t2} \| u ( 0 , \bullet ) \|^2_{ L^2 ( \T^2) } \leq K  \int_0^T \|
u ( t , \bullet )  \|^2_{ L^2 ( \Omega ) } dt \,. 
\end{equation}
\end{thm}
An estimate of this type is called {\em an observability}  result.
Once we have it, the  HUM
method  (see~\cite{Li}) automatically provides the following {\em
  control} result:
\begin{thm}
\label{t:3}
Let $ \Omega \subset \T^2$ be any nonempty open set and let $ T > 0 $. For any $u_0 \in L^2(\T^2)$, there exists $f\in L^2([0,T]\times \Omega)$ such that the solution of the equation
$$ (i\partial_t +\Delta - V(z) ) u (t,z)= f \bbbone_{[0,T] \times
  \Omega}(t,z) \,,  \qquad  u ( 0 , \bullet )  = u_0\,, $$
satisfies 
$$ u ( T , \bullet )  \equiv 0 \,. $$
\end{thm}

In the case of $ V \equiv 0 $ (and rational tori) 
the estimates \eqref{eq:t1} and~\eqref{eq:t2} were proved
by Jaffard \cite{Ja} and 
Haraux \cite{Ha} using Kahane's work \cite{Ka}
on lacunary Fourier series.  For $ V \in \CI ( \T^2 ) $ the
results above were proved by the last two authors \cite{BZ4} 
and for a class potentials including continuous potentials on $ \T^n $,
by Anantharaman-Macia \cite{AM}. The paper \cite{AM}
resolves other questions concerning semiclassical measures on 
tori and contains further references; see also \cite{B}.
For a presentation of other aspects of control theory for the Schr\"odinger equation
we refer to \cite{Le} -- see also \cite[\S 3]{BZ2}.

The paper is organized as follows. In \S \ref{apri} we present
dispersive estimates which allow approximation of rough
potentials by smooth potentials.  In \S \ref{odo} we refine some of the one
dimensional observability estimates and show that they hold
for potentials $ W \in L^p ( \T^1 ) $, $ p > 1$.  The next
\S \ref{semi} is devoted to semiclassical observability estimates
for a family of smooth potentials compact in $ L^2 ( \T^2 ) $. In the following
section an observability result is proved for
general tori with constants uniform in a compact set in $ L^2 $ (Proposition \ref{th.1} i).
Combined with the results from \S \ref{apri} that gives
the proof of the theorem.

\section{A priori estimates for solutions to Schr\"odinger equations}
\label{apri} 
The proof of observability for rough potentials will follow from
observability for smooth potentials with estimates controlled by
constants depending only on $ L^2 $ norms of the potential. 
The approximation argument uses dispersion estimates for the
Schr\"odinger goup on the torus and 
we first show that these estimates hold in the presence of a potential.

\subsection{The case of $ \T^1$.} 
\label{cd1}
We start with the simpler case of
one dimensional equations. It will be needed in \S \ref{odo} but it also
introduces
the idea of the proof in an elementary setting.

We first  make some general comments. The operator $ -\partial_x^2 + W $, $ W
\in L^1 ( \T^1 ) $ is defined by Friedrich's extension (see for
instance \cite[Theorem 4.10]{DiSj}) using the quadratic form 
\[   q ( v , v ) = \int_{\T^1  }  \left( | \partial_x v ( x )  |^2  +
  W ( x ) | v ( x ) |^2 \right) dx , \ \ v \in H^1 ( \T^1 ) , \]
which is bounded from below since
\[ \begin{split}  | \int_{ \T^1}  W ( x ) | v ( x)|^2  dx| & \leq C \| W \|_{ L^1 }
\| u \|_{L^\infty }^2 \leq C \| W \|_{ L^1 }
\| \partial_x v \|_{L^2 } \| v \|_{L^2 } \\
&  \leq - C \epsilon 
 \| W \|_{ L^1 }
\| \partial_x v \|_{L^2 }^2 - \frac C  \epsilon \| W \|_{L^1} \| v
\|_{L^2 }^2 . 
\end{split} 
\] 
Hence $ P= - \partial_x^2 + W $ defined on $ C^\infty ( \T^1 ) $ 
has a unique self-adjoint extension with the domain 
containing $ H^1 ( \T^1 ) $. When $ W \in L^2 ( \T^1 ) $ the
operator is self-adjoint with the domain $ H^2 ( \T^1 ) $.
The resolvent, $ ( - \partial_x^2 + W - z )^{-1} $ , $ z \notin \RR $
is compact and the spectrum is discrete with eigenvalues  $\lambda_j 
\to + \infty $.

The following estimate applies to solutions of the Schr\"odinger
equation satisfying Floquet periodicity conditions:
\begin{equation}
\label{eq:Fl} 
v ( x + 2 \pi ) = e^{ 2 \pi i k } v ( x ) ,
\end{equation}
or equivalently to solutions of the Schr\"odinger equation with 
$ \partial_x $ replaced by $ \partial_x + i k $. (We note that
$ u ( x ) := e^{ - i k x } v ( x ) $ is periodic and $ \partial_x v (
x) = e^{i k x } ( \partial_x + ik ) u ( x ) $.)

\begin{prop}
\label{p:WW}
For any  $W \in L^2( \T^1)$, there exists $C>0$ such that for any $k
\in [ 0 , 1) $,  and $ u_0\in L^2( \T^1)$
the solution to the Schr\"odinger equation 
\begin{equation}\label{eq.sch}
(i \partial_t +  (\partial_x + i k )^2 - W) u =0, \qquad v\mid_{t=0} = u_0
\end{equation}
 satisfies 
\begin{equation}\label{eq.est3}
 \| u \|_{L^\infty(\T^1_x; L^2( 0,T))}\leq C(1+\sqrt{T}) (1+ \|W\|_{L^2( \T^1)})\|u_0\|_{L^2( \T^1)}.
\end{equation} 
\end{prop}
\begin{proof} 
For $ W
\equiv 0 $ we put $ T = 2 \pi $ so that
, with $ c_n = \hat u_0 ( n ) $, we have
\begin{equation}
\label{eq:Linf}  \begin{split} 
 \| e^{ i t \partial_x^2}  u_0 \|_{L^\infty_x L^2_t }^2 & = \sup_x
  \int_{0}^{2 \pi} \left|\sum_{ n\in \ZZ } c_n  e^{ - i t |
      n  +  k|^2 + i n x } \right|^2 dt \\
& = \sup_x \sum_{ n,m \in \ZZ } \int_0^{2 \pi} 
e^{ i (|n + k |^2 - |m + k |^2 ) t } e^{ i ( n - m ) x }  c_n \bar c_m
dt
\\ & =  \sup_x \sum_{ n \in \ZZ }  \left| \sum_{ \stackrel{ m \in \ZZ }{ \pm
    ( m + k ) =  n + k } }
c_{m} e^{i m x } \right|^2 \leq 
4 \sum_{n \in \ZZ }  | c_n|^2 \leq C \| u_0\|_{L^2 ( \T^1 ) } . 
\end{split}
\end{equation}
(We note that $ \pm ( m + k ) = n + k $ has one solution 
only when $ k \neq 0 , \frac12 $ and two solutions
$ m = \pm n $ for $ k = 0 $ and $ m = n, -n-1$ for $ k = \frac 12 $.)
For a non-zero potential $ W \in L^2 ( \T^1 ) $ we use 
Duhamel's formula and write
\[ u ( t ) = e^{it \partial_x^2} u_0+ \frac 1 i \int_0^t
e^{i(t-s) \partial_x^2 }\left( W u ( s ))\right) ds .\]
Applying \eqref{eq:Linf} (now with a small  $  T> 0  $) and the
Minkowski inequality we obtain
\begin{equation}
\label{eq:Mink} \begin{split}
 \|u\|_{L^\infty_x L^2_t ( [ 0 , T ])  }& \leq C  \| u_0\|_{L^2_x} +\int_0^{T} \|
 \bbbone_{s<t}e^{i(t-s) \Delta } (W u (s))  \|_{L^\infty_x L^2_s ( [ 0 , T ] ) } ds \\
&  \leq  C  \| u_0\|_{L^2_x} + \int_0^{T} \|  e^{i(t-s) \Delta } ( W
 u ( s )  \|_{L^\infty_x L^2_s ( [ 0 , T ] )  }ds\\
&  \leq C  \| u_0\|_{L^2_x} + C \int_0^{T} \|  W u ( s ) \|_{L^ 2_x } ds 
\\ & \leq C  \| u _0 \|_{ L^2_x } + C \sqrt T \| W \|_{ L^2} \| u \|_{  L^\infty_x L^2_t ( [ 0 , T ] ) }.
\end{split}
\end{equation}
Hence 
\begin{equation}
\label{eq:WW}   \|u\|_{L^\infty_x L^2_t ( [ 0 , T ])  } \leq 2 C \| u \|_{
  L^2_x } , \ \ \text{if}  \ \ 
 \sqrt T \| W \|_{ L^2 } \leq \frac1{4 } . \end{equation}
To obtain the estimate for multiples of $ T $ satisfying \eqref{eq:WW} 
we note that, by the invariance of the $ L^2_x $ norm of $ u ( t ) $,
$  \int_{(k-1)T}^{kT} \|u ( t ) \|_{L^\infty_x }^2 dt  \leq 2C  \| u(
(k-1) t ) \|_{L^2_x} = 2 C \| u_0 \|_{ L^2_x} $.
Iterating this inequality gives \eqref{eq.est3}.
\end{proof}


\subsection{The case two dimensional tori.}
\label{cd2}
We now assume that $ A = 2\pi , B = 2\pi \gamma ^{-1}>0  $ in the definition of $ \T^2
$. The  case of general $ A , B $ follows by rescaling. 
For $n= (n_1, n_2) \in \Z^2$, we shall denote by 
\begin{equation}
\label{eq:dot}  |n| = \sqrt{ n_1^2 + \gamma n_2^2}, \ \ \ n \cdot x =
n_1 x_1 + \gamma n_2 x_2.
\end{equation}

We start with some general
observations. If $ V \in L^2 ( \T^2 ; \RR ) $ then $ - \Delta + V $
on $ \CI ( \T^2 ) $ is a symmetric operator. Also, by Sobolev
inequalities, 
\[  ( - \Delta + i )^{-1} :  L^2 ( \T^2) \rightarrow  H^2 ( \T^2 )
\hookrightarrow  C^{0,1-} ( \T^2 ) \hookrightarrow L^\infty ( \T^2 ) , \]
is a compact operator. Hence, 
as the multiplication by $ V \in L^2 $ is bounded $ L^\infty \to L^2 $, 
$ V ( - \Delta + i )^{-1} $ is a compact operator on $ L^2 $.
It follows that the operator $ - \Delta  + V $ is essentially self-adjoint and has
a discrete spectrum (see for instance \cite[Theorem 4.19]{DiSj}).
Since for for $ u \in  H^2 ( \T^2 ) \subset L^\infty ( \T^2) $, $ V u \in
L^2 $, the domain is equal to $ H^2 ( \T^2 ) $. In particular, 
\[ u ( t ) := e^{  i t (  \Delta - V ) } u_0 \in C^0 ( \RR_t ; H^2 (
\T^2 ) ) \cap C^1 ( \RR_t ; L^2 ( \T^2 ) ) , \]
and
\begin{equation}
\label{eq:Duh}  u ( t ) = e^{ it \Delta } u_0 + \frac{1}i \int_0^t
e^{ i ( t - s) \Delta } ( V  u ( s ) ) ds .
\end{equation}

\begin{prop}\label{lem.zyg}
Let $T>0$. For any compact subset $\mathcal{V} \subset L^2( \T^2)$,
there exists $C( \mathcal{V}),\epsilon >0$ such that for any 
\[ V\in\mathcal{V} + B(0, \epsilon)\subset  L^2( \T^2)\]
and any 
\[  v_0\in L^2( \T^2), \ \ \ f \in L^1((0, T); L^2( \T^2))+ L^{\frac 4
  3 } (\T^2 ; L^2(0, T)),\]
 the solution to
\begin{equation}\label{eq.schpot}
(i \partial_t + ( \Delta - V)) u =f, \qquad u\mid_{t=0} = v_0, 
\end{equation}
 satisfies 
\begin{multline}\label{eq.est2}
 \| u\|_{L^\infty([0,T]; L^2( \T^2)) \cap L^4(\T^2_x; L^2( 0,T))}\\
 \leq C( \mathcal{V}) \Bigl(\|v_0\|_{L^2( \T^2)} + \|f\|_{ L^1((0, T); L^2( \T^2))+ L^{\frac 4 3 } (\T^2 ; L^2(0, T))}\Bigr).
\end{multline} 
\end{prop}
Before proving this result, let us show how it implies that 
Jaffard's result (Theorem \ref{th.2} with $ V = 0 $) is stable by perturbation with potentials small in $L^2( \T^2)$:
\begin{cor}\label{cor.perturb}
For any non-empty open sent $ \Omega $ and $ T > 0 $, there exist constants $\kappa, K>0 $ such that for
 $V\in L^2(\T^2)$, 
\[ \| V\|_{L^2( \T^2)} \leq \kappa \ \Longrightarrow \ 
\| u_0 \|^2_{ L^2 ( \T^2) } \leq K  \int_0^T \|
e^{ - i t ( - \Delta + V ) } u_0  \|^2_{ L^2 ( \Omega ) } dt \,, \]
for any $ u_0 \in L^2 ( \T^2 ) $.  
\end{cor}
\begin{proof} The Duhamel formula gives
$$ u = e^{- it (- \Delta + V) } u_0 = e^{it\Delta}u_0  + \frac 1 i
\int_0^t e^{i(t-s) \Delta} \left( Vu(s) \right) ds, 
$$ 
and Jaffard's result (estimate~\eqref{eq:t2} for $V=0$) applies to the
first term. Hence, for a constant $ K_0 $ depending  on $ \Omega$ and
$ T $, 
\begin{equation}\label{eq.perturb}
 \begin{split}
 \|u_0\|_{L^2 ( \T^2 ) } & \leq K_0 \int_0^T \| e^{it \Delta} u_0  \|^2_{ L^2 ( \Omega ) } dt  
\\ & =  K_0 \int_0^T \|
\textstyle{ e^{it( \Delta - V )} u_0  - \frac 1 i  \int_0^t e^{i(t-s)
    \Delta} ( Vu(s)) ds } \|^2_{ L^2 ( \Omega ) } dt \\
&\leq 2 K_0 \|
e^{it( \Delta - V )} u_0 \|^2_{ L^2 ( \Omega ) } dt + 2K_0 T
\|\textstyle{ \int_0^t e^{i(t-s) \Delta} ( Vu(s))
  ds}\|_{L^\infty([0,T]; L^2( \T^2))} ^2 .
\end{split}
\end{equation}
We now use Proposition \ref{lem.zyg} with $ {\mathcal V} = \{ V \} $,  $ v_0 = 0 $ and $ f = V u $
to obtain
\[ \|\textstyle{ \int_0^t e^{i(t-s) \Delta} ( Vu(s))
  ds}\|_{L^\infty([0,T]; L^2( \T^2))} \leq C \|V u \|_{L^{\frac 4 3
  } (\T^2 ; L^2(0, T))}^2 \leq C \| V \|_{L^2 ( \T^2 )} \| u \|_{ L^4
  ( \T^2 , L^2 ( 0 , T ) )}  .\]
Applying Proposition \ref{lem.zyg} to the righthand side, now with $ v_0 =
u_0 $, $ f = 0 $, gives 
\[  \|\textstyle{ \int_0^t e^{i(t-s) \Delta} ( Vu(s))
  ds}\|_{L^\infty([0,T]; L^2( \T^2))} \leq C \| V \|_{L^2 ( \T^2 ) }
\|u _0 \|_{L^2} ,\]
so that \eqref{eq.perturb} becomes
\[ \| u\|_{L^2 ( \T^2)} \leq 2 K_0 \|
e^{it( \Delta - V )} u_0 \|^2_{ L^2 ( \Omega ) } dt + 2 C K_0 T
\|V\|_{L^2( \T^2)}^2 \| u_0 \|_{L^2(\T^2))}^2 .
\]
To conclude, it sufficies to take $2 C K_0 T \kappa^2  \leq
1/2$.  (We note that since $K_0$ depends on $\Omega$ and $ T$ while $C$ depends
on $T$, we have no other choice than taking $\kappa >0$  small.)
\end{proof}

\medskip
\noindent
{\bf Remark.} In \S \ref{frsm} we will eliminate the smallness
assumption on $ \| V \|_{L^2} $ and that will prove Theorem
\ref{th.2}.

\medskip

The proof of Proposition \ref{lem.zyg} proceeds in several steps. 
We start proving estimate for $V=0$, then we prove the general case by
a perturbation arguments.

The next proposition is a ``fuzzy'' version of the classical estimate
of Zygmund:
\begin{equation}
\label{eq:zyg0} \exists \, C>0 \; \forall \, \tau \in \N, \ \ \| \sum_{{n\in
  \Z^2},\, { |n|^2= \tau}} c_ne^{in \cdot x} \|_{L^4( \T^2)}^2 \leq C
\sum_{{n\in  \Z^2},\, { |n|^2= \tau}} |c_n|^2 ,
\end{equation}
and it is motivated by the C\'ordoba square function estimate \cite{Co}:

\begin{prop}\label{prop.zygmund}
There exists $C>0$ such that for any $  0 \leq \kappa $ and $ 0 < h < 1 $, 
and any $ u
\in L^2( \TT^2)$ satisfying
\[  \hat u ( n ) = 0  \ \text{ for } \ n \notin {\mathcal B} ( \kappa,
h ) := \{ n\in \ZZ^2; | h^2 |n|^2 - 1 | \leq
 \kappa^2 h^2 \} . \]
we have 
\begin{equation}\label{zygmund}
\| u \| _{L^4( \TT^2)} \leq C  ( 1 + \kappa) ^{\frac12}  \| u\|_{L^2( \TT^2)}
\end{equation}
\end{prop}
We note that \eqref{eq:zyg0} is the case of $ \kappa = 0 $.

\begin{proof}
We first note that we can assume that $ \kappa \geq 1 $ as the
sets $ {\mathcal B} ( \kappa, h ) $ increase with increasing $ \kappa $.

For a constant $ \delta > 0 $, to be fixed later, 
we distinguish two regimes: $ \kappa h  \geq \delta $ and $ \kappa h \leq
\delta $. 
In the first regime, the estimate follows from the Sobolev embedding
$H^{\frac12} ( \TT^2) \rightarrow L^4(\TT^2)$:  $ \hat u ( n ) = 0 $ unless
$ |n|^2 \leq h^{-2}  + \kappa^2 \leq (1/\delta +1 ) \kappa ^2
$, and this implies 
$$ \| u\| _{H^{\frac12} ( \TT^2)} \leq C_\delta \kappa^{\frac 1 2}\| u\|_{L^2}$$

From now on we assume that $ h \kappa \leq \delta $.  In this regime, we can change the set $\mathcal{B} (\kappa, h) $ to
$${\mathcal A} ( \kappa,
h ) := \{ n\in \ZZ^2; | h |n| - 1 | \leq
 \kappa^2 h^2 \} . 
 $$The idea is to
prove  an {\em arithmetic  version of the C\'ordoba square function
  estimate} \cite{Co}. Indeed, the usual version allows only to work with $\kappa \geq h^{-\frac12}$ 
(the uncertainty principle). Our version below allows to get estimates all the way down to $\kappa \sim 1$ 
(that is,  much beyond the uncertainty principle). 
We first notice that we can also assume that the spectrum of $u$ is also contained in the upper quadrant of the plane 
$ \{ z \in \CC : \Re z \geq 0 , \Im z \geq 0 \}$ (here and in what
follows we identify $ \RR^2 $ with $ \CC $).
Indeed, if the result is true for the upper quadrant, by symmetry, it is true
for any quadrant, and, with a different constant in the general case.
Then we decompose the intersection of the annulus with this quadrant
into a disjoint union of angular sectors of angles 
$ h  \kappa  $:
$$  \mathcal{A}( \kappa , h )  \cap \{ \Im z \geq 0 , \ \Re z \geq 0 \}  = \bigcup_{\alpha =0}
^{N_{\kappa , h }}  \mathcal{A}_{\alpha} ( \kappa, h ) ,  \ \ \  N_{ \kappa, h } :=
\left[   
\frac{ \pi} { 2 h  \kappa  } 
\right]
,  $$
where 
$$ \mathcal{A}_{\alpha}( \kappa, h ) := \{ z : \Re z \geq 0, \ \Im z \geq 0, \  \ 
| h |z| - 1 | \leq  \kappa^2 h^2 , \ \ \arg(z) \in [\alpha h
 \kappa  , (\alpha+1) h  \kappa   )  \}
$$ 
The proof relies on the following geometric lemma which will be proved
Appendix~\ref{app.B}:
\begin{lem}\label{geom}
Fix $\delta>0$ small enough. Then there exists $Q\in \NN$ such that
for any $ 0 < h < 1 $, any $1 \leq \kappa \leq \delta/h $, we have 
\begin{equation}
\begin{gathered}
\forall \alpha, \beta, \alpha ', \beta ' \in \{ 0, 1, \dots,
N_{\kappa, h }\} ^4,\\
 (\mathcal{A}_{\alpha} ( \kappa, h ) +\mathcal{A}_{\beta}( \kappa, h
 ) )\cap (\mathcal{A}_{\alpha'} ( \kappa, h
 )  +\mathcal{A}_{\beta '} ( \kappa, h
 )  ) \neq \emptyset\\
\Longrightarrow | \alpha - \alpha'| + | \beta- \beta ' |\leq Q \
\text{ or } \ | \alpha - \beta'| + | \beta- \alpha ' | \leq Q 
\end{gathered}
\end{equation}
\end{lem}

We apply the lemma as folllows.
We have  
$$ u= \sum_{\alpha =0}^{N_{\kappa, h}} U_{\alpha}, \ \ \ 
 u^2 = \sum_{\alpha, \beta  =0}^{N_{\kappa, h} }U_{\alpha}
 U_{\beta}, \ \ \ 
U_\alpha : = \sum_{\ZZ^2 \cap \mathcal{A}_{\alpha}} u_{n} e^{in\cdot x}.$$
and hence
\begin{equation}\label{L4}
 \| u\|_{L^4(\TT^2)} ^4 = \sum_{ \alpha, \beta, \alpha ', \beta '   =0}^{N_{\kappa, h}}\int_{\TT^2} U_{\alpha} U_\beta \overline{U} _{\alpha '} \overline{U} _{\beta'}(x) dx
\end{equation}
The integral vanishes unless 
$$(\mathcal{A}_{\alpha} ( \kappa, h ) +\mathcal{A}_{\beta}( \kappa, h
)  )\cap (\mathcal{A}_{\alpha'} ( \kappa, h ) +\mathcal{A}_{\beta '}(
\kappa, h )  ) \neq \emptyset $$
as otherwise 
$$n\in \ZZ^2 \cap \mathcal{A}_{\alpha}, \ \ m \in \ZZ^2 \cap
\mathcal{A}_{\beta}, \ \ p\in \ZZ^2 \cap \mathcal{A}_{\alpha' }, \ \ q\in
\ZZ^2 \cap \mathcal{A}_{\beta'} \  \Longrightarrow \  n+m- (p+q) \neq 0
,$$ 
and, using the inner product \eqref{eq:dot},   $\int_{\TT^2} e^{ix\cdot (n+m- p-q)} dx =  0$.
Lemma~\ref{geom} then shows that we can restrict the sum
in~\eqref{L4}  to the subset of indexes $(\alpha, \beta, \alpha',
\beta ')$ satisfying 
\[ | \alpha - \alpha'| + | \beta- \beta ' |\leq Q \ \text{ or } \ | \alpha
- \beta'| + | \beta- \alpha ' | \leq Q . \]
This and an application of H\"older's inequality, 
\[ \begin{split} \Bigl|\int_{\TT^2} U_{\alpha} U_\beta \overline{U}
  _{\alpha '} \overline{U} _{\beta'}(x) dx\bigr| & \leq \|
  U_{\alpha}\|_{L^4( \TT^2)}\| U_{\beta}\|_{L^4( \TT^2)}\|
  U_{\alpha'}\|_{L^4( \TT^2)}\| U_{\beta'}\|_{L^4( \TT^2)} \\
&  \leq \left\{ \begin{array}{l}  \left( \|  U_{\alpha}\|_{L^4( \TT^2)}^2 +  \| U_{\alpha'}\|_{L^4(
  \TT^2)} ^2 \right) \left(  \| U_{\beta}\|_{L^4( \TT^2)}^2 + \|
U_{\beta'}\|_{L^4( \TT^2)}^2 \right) \\ 
\ \\
\left( \|   U_{\alpha}\|_{L^4( \TT^2)}^2 +  \| U_{\beta'}\|_{L^4(
  \TT^2)} ^2 \right) \left(  \| U_{\beta}\|_{L^4( \TT^2)}^2 + \|
U_{\alpha'}\|_{L^4( \TT^2)}^2 \right) \end{array} \right. 
\end{split} ,\]
give
\begin{equation}\label{eq.10}  \| u\|_{L^4( \TT^2)} ^4 \leq C Q^2 \Bigl(\sum_{\alpha =0 } ^{N_{\kappa, h}}\| U_{\alpha}\|_{L^4( \TT^2)}^2 \Bigr) ^2.
\end{equation}
To estimate the norms of $ U_\alpha $ we write 
\begin{equation}\label{eq.1}
  \begin{split}
\| U_{\alpha}\|_{L^4  ( \TT^2)}&\leq C \| U_{\alpha}\|^{1/2}_{L^\infty (
  \TT^2)} \| U_{\alpha}\|^{1/2}_{L^2(
  \TT^2)} \\
   &\leq \bigl(\sum_{n\in \ZZ^2 \cap \mathcal{A}_\alpha} |u_n| \bigr)^{1/2} \bigl(\sum_{n\in \ZZ^2 \cap \mathcal{A}_\alpha} |u_n|^2\bigr)^{1/4}
  \leq C  | \ZZ^2 \cap \mathcal{A}_{\alpha} ( \kappa , h ) |^{\frac14} \| U_{\alpha}\|_{L^2( \TT^2)} .
\end{split}
  \end{equation}
  To estimate the number of integral points  in $\mathcal{A}_{\alpha} (
  \kappa, h ) $, we
  first notice that 
 $\mathcal{A}_{\alpha} ( \kappa , h ) $ is included in a rectangle of height $ 1  +
 {\kappa} $ and width $ 1 + 3 {\kappa}^2 h $.

\begin{figure}[ht]
\includegraphics[width=4.5cm]{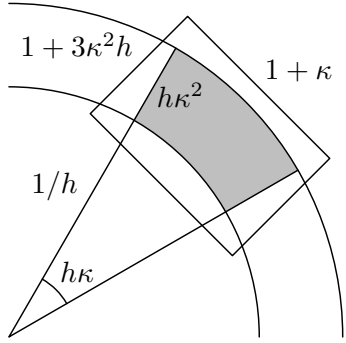}
\caption{The angular region $ \mathcal A_\alpha ( \kappa , h ) $
fitted inside a rectangle.}
\label{f:1}
\end{figure}

Now, the number of
 integral   points in any rectangle of height
 $H$ and width $W$ is bounded by $C \max(H,1)
 \max(W,1)$. (To see this, notice that open discs of radius $\frac12$
 centered at the integer points are pairwise disjoint and are all included in
 a rectangle of height $H+1$ and width $W+1$.)  Hence, recalling that 
$ \kappa h \leq \delta $, 
\[  | \ZZ^2 \cap \mathcal{A}_{\alpha} ( \kappa , h ) |  \leq C ( 1 +
\kappa ) ( 1 + 3 \kappa^2 h )  \leq C ( 1 + \kappa)^2 . \]
Combining this with \eqref{eq.1} and \eqref{eq.10} gives
  $$ \| u\|_{L^4 ( \T^2 ) } ^4 \leq C ( 1 + \kappa )^2 \| u\|_{L^2 (
    \T^2) }^4, $$ 
concluding the proof. 
\end{proof} 

The next step in the proof of Proposition \ref{lem.zyg} is an optimal (at least in terms of the spectral region where it holds)
resolvent estimate -- see Kenig-Dos Santos-Salo~\cite[Remark
1.2]{KDSS} and Bourgain-Shao-Sogge-Yao~\cite{BSSY}
for related results. 
\begin{prop} \label{theorem.resol}
For any compact subset  $\mathcal{V}\subset L^2( \T^2)$, there exists $C(\mathcal{V}), \epsilon >0$ such that for any  $V\in \mathcal{V}+ B(0, \epsilon)$, any $ f \in C^\infty( \T^2)$ and any $ \tau \in \C, |\Im \tau| \geq 1$,  
 \begin{equation}
\label{eq.resolbis}
 \| ( - \Delta+V - \tau)^{-1} f \|_{L^4 ( \T^2)} \leq C  \| f \|_{L^{4/3}( \T^2)}
 \end{equation}
 \end{prop}
We deduce it from Proposition \ref{prop.zygmund} and the 
following elementary result:
 \begin{lem}\label{lem.decomp} Assume that $\mathcal{V}$ is a compact subset of $ L^2( \T^2)$. Then for any $\delta >0$ there exists $C_\delta>0$ and for any $V\in \mathcal{V}$ there exists $V_ \delta \in L^\infty ( \T^2)$ such that 
 $$ \| V_\delta - V\|_{L^2( \T^2)} \leq \delta, \qquad \|V_\delta \|_{L^\infty ( \T^2)} \leq C_\delta.$$
 \end{lem}
\begin{proof} This is obvious for $ {\mathcal V } = \{ V_0 \} $ since $
  L^\infty \subset L^2 $ is dense. Applying it with $ \delta $
  replaced by $ \delta/2 $ the statement remains true for $ V $ with $
  \| V - V_0 \|_{L^2} \leq \delta/2 $. A covering arguments provides
  the result for a general compact set in $ L^2 $. 
\end{proof}

 \begin{proof}[Proof of Proposition~\ref{theorem.resol}]
 For $\Re \tau \leq C$  for any fixed $ C $, we get~\eqref{eq.resolbis} directly. Indeed, from 
$( - \Delta -\tau +V) u =f$, multiplying by $\overline{u}$, integrating by parts and taking real and imaginary parts, we get 
\begin{equation*}
\begin{aligned}
 \| \nabla u \|_{L^2( \T^2)}^2 - \Re \tau \| u\|_{L^2 ( \T^2)}^2 &\leq    \| V |u|^2\|_{L^1( \T^2)}+  \| u\|_{L^4( \T^2)} \| f\|_{L^{4/3} ( \T^2)}, \\
   |\Im \tau| \| u\|_{L^2 ( \T^2)}^2 &\leq   \| u\|_{L^4( \T^2)} \|
   f\|_{L^{4/3} ( \T^2)} .
   \end{aligned}
   \end{equation*}
Since  $|\Im \tau| \geq 1$, the Sobolev embedding and Lemma
\ref{lem.decomp} imply
\[ \begin{split}
\| u \| _{L^4( \T^2)} ^2 & \leq C \| u \|_{H^1( \T^2)}^2 \\
& \leq C \bigl( \| V_\delta - V \| _{L^2( \T^2)} \| u\|^2_{L^4( \T^2)}  + \| V_\delta\|_{L^\infty( \T^2)} \| u \|_{L^2( \T^2)} ^2+ \| u\|_{L^4( \T^2)} \| f\|_{L^{4/3} ( \T^2)}\bigr),\\
& \leq C ( \delta + \epsilon ) \|  u\|^2_{L^4( \T^2)} +C \bigl(\| V_\delta\|_{L^\infty( \T^2)}+1) \| u\|_{L^4( \T^2)} \| f\|_{L^{4/3} ( \T^2)}\bigr) 
\end{split} \]
and  choosing $  \epsilon < \delta =  \frac14 C$ gives the result. 

For  $\Re \tau > C $ we start with the case of $ V = 0 $ and notice
\[  ( - \Delta - \tau )^{-1} = ( - \Delta - \tau)^{-\frac12} \left( ( -
  \Delta - \bar \tau )^{-\frac12} \right)^* : L^{\frac 43} \longrightarrow L^4  \]
follows from $ ( - \Delta - \tau)^{-\frac12} : L^2 \to L^4 = (
  L^{\frac43})^* $. Here the square root is defined using the spectral
  theorem and the branches chosen  for  $ \pm \Im \tau > 1 $  so that 
$$   (  \lambda -   \tau)^{\frac12} \overline{ ( \lambda - \bar \tau
    )^{\frac12}}  = \lambda   - \tau , \ \  \lambda \geq 0 . $$

Hence we need to prove that
$$  \|u \|_{L^4 ( \T^2)} \leq C \| f \|_{L^{2}( \T^2)}, \ \ \ u :=  (
- \Delta - \tau)^{-\frac12} f .
$$
To use Proposition~\ref{prop.zygmund} we write the resolvent applied
to $ f $ using the Fourier series:
\[  u = \sum_{n} \frac{ f_n} { (|n|^2 - \tau)^{\frac12}} e^{in \cdot x} =
u_0 + \sum_{j=1}^\infty  u_j , \ \ \ u_j :=  \sum_{ 2^{j-1} \leq |
  |n|^2 - {\Re \tau}  | <  2^j}\frac{ f_n} {
  (|n|^2 - \tau)^{\frac12}} e^{in \cdot x} . 
\] 
We note that $ u_0 = \sum_{ ||n|^2 - \Re \tau | < 1 } f_n  ( |n|^2 -
\tau)^{-\frac12} e^{ in\cdot x } $ and hence Proposition
\ref{prop.zygmund} gives 
\[ \| u_0\|_{ L^4 ( \T^2 ) } \leq C \| f
\|_{L^2 ( \T^2 ) } .\]
Applying \eqref{zygmund} to $ u_j$'s, with 
$ h = (\Re \tau)^{-\frac12}$ and $ \kappa = 2^{j/2} $ 
gives
\[ \begin{split} 
  \| u -u_0 \|_{L^4( \T^2)} & \leq C \sum_{j} 2^{j/4}
   \| u_j \|_{L^2} 
 \leq \bigl(\sum_{j=1}^\infty 2^{-j/2}\bigr)^{\frac12}\Bigl(\sum_{j=1}^\infty 2^{j} 
\sum_{2^{j-1} \leq | |n|^2 -  \Re \tau |< 2^j} \frac{ |f_n|^2} { ||n|^2
  - \tau|}\Bigr)^{\frac12} \\ 
& \leq C \| f\|_{L^2}
\end{split} \]
which concludes the proof of Proposition~\ref{theorem.resol} for
$V=0$. 

The general case $V\neq 0$ follows from the same perturbation argument
as in the case $\Re \tau \leq C $. Indeed, from 
$ (- \Delta - \tau )u = -V u +f,$ 
we deduce 
$$  |\Im \tau| \| u\|_{L^2 ( \T^2)}^2 \leq   \| u\|_{L^4( \T^2)} \| f\|_{L^{4/3} ( \T^2)},$$
and from the resolvent estimate for $V=0$,
\[ \begin{split} 
  \| u \| _{L^4( \T^2)} &  \leq C \| V u\| _{L^{4/3}( \T^2)}+ \| f\|_{L^{4/3}( \T^2)} \\
& \leq C \bigl( \| V_\delta - V \| _{L^2( \T^2)} \| u\|_{L^4( \T^2)}  + \| V_\delta\|_{L^\infty( \T^2)} \| u \|_{L^2( \T^2)} +  \| f\|_{L^{4/3} ( \T^2)}\bigr),\\
& \leq C \delta \|  u\|_{L^4( \T^2)} +C \bigl(\| V_\delta\|_{L^\infty( \T^2)} \| u\|^{\frac12}_{L^4( \T^2)} \| f\|^{\frac12}_{L^{4/3} ( \T^2)} + \| f\|_{L^{4/3} ( \T^2)}\bigr) .
\end{split} \] 
Choosing $ \delta $ small enough gives the desired estimate.
\end{proof}

\begin{proof}[Proof of Proposition \ref{lem.zyg}] Let us first study the contribution of $v_0$. 
Putting $Tu_0 = e^{it( \Delta - V)} u_0$ we have 
$$ TT^* f= \int_0 ^T e^{i(t-s) ( \Delta -V)} f(s) ds= \int _{0}^te^{i(t-s) ( \Delta -V)} f(s) ds+  \int _{t}^Te^{i(t-s) ( \Delta -V)} f(s) ds.
$$
To prove that $ T: L^2 ( \T^2 ) \to L^4 ( \T^2_x , L^2 ( [0, T]) )$
it suffices to prove that 
$$ T T^* : L^{\frac43} ( \T^2_x , L^2 ([0,T]) ) \to L^4 ( \T^2_x , L^2
( [0, T]) ), $$
 and we will show it for
the two operators on the right hand side, say the first one.  That means showing that for
solutions to $ ( i \partial_t + \Delta - V ) v = f $, $ v|_{t=0} = 0
$, we have 
\begin{equation} \label{eq.dispersion}
\| v\|_{L^4( \T^2; L^2[0,T])} \leq C \| f \|_{L^{4/3}( \T^2; L^2[0,T])}.
 \end{equation}
Let $U = v e^{-t}\bbbone_{t>0}, F= fe^{-t} \bbbone_{0<t<T}$. We have 
$ (i\partial_t + \Delta - V +i ) U = F$ and hence by taking the
Fourier transform in  $ t $, 
$$ ( \Delta -V +i - \tau) \widehat{U} = \widehat{F}.$$
Proposition~\ref{theorem.resol} now shows that for any $\tau \in \R$, 
 $$ \| \widehat{U} (\tau)\|_{L^4( \T^2)} \leq C \| \widehat{F} (\tau)\|_{L^{4/3}( \T^2)},
 $$ which implies 
 \begin{equation}
 \begin{aligned}
 \| u\|_{L^4(\T^2_x; L^2( 0,T)) } &\leq C \|U\|_{L^4(\T^2_x; L^2(\R_t)) }= C \|\widehat{U}\|_{L^4(\T^2_x; L^2(\R_\tau)) }\\
 &\leq C \|\widehat{U}\|_{L^2(\R_\tau;L^4(\T^2_x )) }
 \leq C'\|\widehat{F}\|_{L^2(\R_\tau;L^{4/3}(\T^2_x )) }\\
 &\leq C' \|\widehat{F}\|_{L^{4/3}(\T^2_x ; L^2(\R_\tau)) }
 = C' \|F\|_{L^{4/3}(\T^2_x ; L^2[0,T]) }
 \end{aligned}
 \end{equation}
concluding the proof of \eqref{eq.dispersion}.

Part of nonhomogeneous estimate in \eqref{eq.est2}, 
$$\| v\|_{L^\infty([0,T]; L^2( \T^2)) \cap L^4(\T^2_x; L^2( [0,T])}\leq C  \|f\|_{ L^1([0,T]; L^2( \T^2))}.
 $$ 
follows from the boundedness of the operator $T$ from $L^2$ to
$L^4(\T^2; L^2([0,T])$ and the Minkovski inequality.  Finally, since the dual of the operator 
 $ f \mapsto \int_0 ^t e^{i (t-s) \Delta -V} f(s) ds$
 is $ g\mapsto \int_t ^T e^{i (t-s) \Delta -V} g(s) ds,$  we also get 
  $$\| u\|_{L^\infty([0,T]; L^2( \T^2)) }\leq C  \|f\|_{L^1([0,T];
    L^2( \T^2))+ L^{\frac 4 3 } (\T^2 ; L^2[0,T])}, 
 $$ 
which concludes the proof of Proposition~\ref{lem.zyg}.
\end{proof}

We conclude this section with a continuity result which will be useful later:
 \begin{prop}
\label{p:VV}
Consider  a sequence,  $ \{ V_n\}_{n\in \N} \subset 
L^2 (\T^2)  $ converging to $V\in L^2( \T^2)$. 
Then there exists $C>0$ such that for any $v_0\in L^2( \T^2)$, 
\begin{equation}
\label{eq.diff}\| e^{-it (-\Delta + V)}v_0  -e^{-it(- \Delta + V_n)} v_0\|_{L^\infty([0,T]; L^2( \T^2))}  \leq C \| V-V_n\|_{L^2( \T^2)} \| u_0 \|_{L^2( \T^2)} 
 \end{equation}
 \end{prop}

\medskip
\noindent
{\bf Remark.}
The result in Proposition~\ref{p:VV} can be stated more generally: 
for a compact subset of $\mathcal{V}\subset L^2( \T^2)$ and is equivalent to the Lipschitz continuity of the map
$$V\in \mathcal{V}\subset L^2( \T^2) \longmapsto e^{-it (-\Delta + V)}\in L^\infty((0,T); \mathcal{L}(L^2( \T^2)).$$
A slight modification of the proof presented here shows that it is in fact also Lipschitz on bounded subsets of $L^p$, $ p>2$.
It would be interesting to investigate such properties on other manifolds, as they seem to depend strongly on the geometry. Indeed, the analysis in~\cite[Theorem 2]{BuGeTz} is likely to give that 
on spheres,  there exists a sequence of potentials $\{V_n\}_{n\in
\NN} $ such the that for any $T>0$, any $p<+\infty$, 
$$\lim_{n\rightarrow + \infty } \| V_n\|_{L^p( \mathbb{S}^2)} =0, \quad \text{ but }  \lim_{n\rightarrow + \infty }\| e^{it \Delta} - e^{it (\Delta - V_n) }\|_{L^\infty((0,T);\mathcal{L}(L^2 ( \mathbb{S}^2)))} >0.
$$

\medskip

\begin{proof}[Proof of Proposition~\ref{p:VV}]
Let $u= e^{it (\Delta -V)}v_0 $ and $u_n= e^{it (  \Delta -V_n)}v_0$,
so that the Duhamel formula gives
$$ u- u_n  = 
\frac 1 i \int_0^t e^{i(t-s) ( \Delta -V)} (V_n- V) u_{n}(s) ds. 
$$
Proposition \ref{lem.zyg} applied
with ${\mathcal V} = \{ V\}$, $ v_0=0$ and $ f= (V_n- V) u_{n}$,
and H\"older's inequality give
\[  \begin{split}
  \| u_V - u_{n}\|_{L^\infty ([0,T] ; L^2(\T^2_x))}  & \leq
  C \| (V-V_n) u_{n} \|_{L^{4/3}(\T^2; L^2([ 0, T])}\\
&  \leq  C  \| (V-V_n)\|_{L^2}
\|u_{n}\|_{L^4( \T^2; L^2([0,T] ))} .\end{split}  \]
Applying Proposition~\ref{lem.zyg} again, now with
$ {\mathcal V} = \{ V_n, n\in \N\} \cup \{V\} $, 
and $ f=0$, we estimate the right hand side to 
obtain the desired estimate:
\[  \| u_V - u_{n}\|_{L^\infty ([0,T] ; L^2(\T^2_x))}
 \leq C \| V-V_n\|_{L^2 ( \T^2 )} \| v_{0} \|_{L^2(\T^2_x ) }.\]
\end{proof}

  \section{One -dimensional observability estimates}
\label{odo}

In this section we consider the one-dimensional analog of our result
which we prove for $L^p$ potentials, $p>1$. 
In applications to control and observability on $2$-tori we will use it only 
it for $p=2$ but the finer estimate may be of independent interest.

Let us make first  some general comments. The operator $ -\partial_x^2 + W $, $ W
\in L^1 ( \T^1 ) $ is defined by Friedrich's extension (see for
instance \cite[Theorem 4.10]{DiSj}) using the quadratic form 
\[   q ( v , v ) = \int_{\T^1  }  \left( | \partial_x v ( x )  |^2  +
  W ( x ) | v ( x ) |^2 \right) dx , \ \ v \in H^1 ( \T^1 ) , \]
which is bounded from below since
\[ \begin{split}  | \int_{ \T^1}  W ( x ) | v ( x)|^2  dx| & \leq C \| W \|_{ L^1 }
\| u \|_{L^\infty }^2 \leq C \| W \|_{ L^1 }
\| \partial_x v \|_{L^2 } \| v \|_{L^2 } \\
&  \leq - C \epsilon 
 \| W \|_{ L^1 }
\| \partial_x v \|_{L^2 }^2 - \frac C  \epsilon \| W \|_{L^1} \| v
\|_{L^2 }^2 . 
\end{split} 
\] 
Hence $ P= - \partial_x^2 + W $ defined on $ C^\infty ( \T^1 ) $ 
has a unique self-adjoint extension with the domain 
containg $ H^1 ( \T^1 ) $. When $ W \in L^2 ( \T^1 ) $ the
operator is self-adjoint with the domain $ H^2 ( \T^1 ) $.
The resolvent, $ ( - \partial_x^2 + W - z )^{-1} $ , $ z \notin \RR $
is compact and the spectrum is discrete with eigenvalues  $\lambda_j 
\to + \infty $.

We have the following one dimensional observability which 
holds for functions satisfying Floquet boundary conditions result:

\begin{prop}\label{prop.1d}
Assume that 
$ W \in L^p ( \T^1 ) $, $p>1$, and $ \omega \subset \T^1 $ 
is a non-empty open set; then for any $ T > 0 $ 
there exist $K_0>0$ such that for any $ k \in [0, 1 ) $ and $ v \in L^2 ( \T^1 ) $, 
\begin{equation}
\label{eq:ob1d-b}
 \| v \|_{L^2 ( \T^1 ) }^2 \leq
{ K_0}   \int_0^T \| e^{ i t (( \partial_x + i k ) ^2 - W) } v \|_{ L^2 ( \omega )}^2 
dt 
\end{equation}
\end{prop}

We first prove the stationary version following the elementary
approach of \cite{BZ3}:
\begin{prop}
\label{pL:1}
Under the assumptions of \eqref{eq:ob1d} there exists
$ C_1 = C_1 ( \omega, \| W \|_{L^p})  $ such that for any $ \tau \in
\RR $, any solution to 
$$  (- (\partial_x + i k )^2 + W -\tau) u = g,$$
\begin{equation}
\label{eq:s1d}
\| u \|_{ L^2 ( \T^1 ) } \leq C_1 \left( 
\langle \tau \rangle^{-\frac12}  \| g\|_{L^2}+ \| u \|_{ L^2 ( \omega ) } \right) .
\end{equation}
\end{prop}

This follows from the following result which holds for $ W = 0 $. 
\begin{lem}
\label{l:old}
Let $ \omega \subset T^1 $ be an open set.
Then there exists a constant $ C_0 = C_0 ( \omega ) $,
 such that that for $ u \in H^1 ( \T^1 ) $ satisfying
\begin{equation}
\label{eq:free}  ( - ( \partial_x + i k ) ^2  - \tau ) u = f + g , 
\end{equation}
we have 
\begin{equation}
\label{eq:lol} 
\begin{split}  \| u \|_{ L^2 ( \T^1 ) } + \langle \tau \rangle^{-\frac12}
\| \partial_x u \|_{L^2 ( \T^2 )}  & \leq 
C_0 \left( \| f \|_{H^{-1} ( \T^1 ) } +  \langle \tau
    \rangle^{-\frac12} \| g \|_{ L^2 ( \T^1 ) }  
+  \| u \|_{L^2 ( \omega ) }  \right) . \end{split}
\end{equation}
\end{lem}
\begin{proof} 
The elementary proof given in \cite{BZ3} shows that if
$ ( - \partial_x^2 - \tau ) u = \partial_x F + G $ then 
\begin{equation}
\label{eq:bsiam} \| u \|_{ L^2 ( \T^1 )}  \leq C \left( \| F \|_{ L^2 ( \T^1 )  } + 
\langle \tau \rangle^{-\frac12} \| G \|_{L^2 ( \T^1 ) } + 
\| u \|_{ L^2 ( \omega ) } \right). 
\end{equation}
We first claim that the result holds when $ \partial_x $
is replaced by $ \partial_x + i k $. Equivalently, that means
that \eqref{eq:lol} holds with $ k = 0 $ for functions which are not 
periodic but satisfy \eqref{eq:Fl}. We will work under the assumption
\eqref{eq:Fl}:
\[  ( - \partial_x^2 - \tau ) v = \partial_x F + G , \ \ v ( x + 2 \pi )
= e^{ 2\pi i k } v ( x) .\]

Choosing a parametrization on $ \T^1 $ so that $ 2 \pi \in \omega $
we take $ \chi \in \CI ( \T^1 ) $ equal to one in a neighbourhood of 
$  \T^1 \setminus \omega $, and vanishing in a neighbourhood of $ 2
\pi$. Hence, $ \supp \chi v \subset ( \epsilon , 2 \pi - \epsilon ) $
and $ u \chi v $ defines a function on  $ \T^1 $.  Applying 
\eqref{eq:bsiam} we obtain, using the properties of $ \chi$, 
\[ \begin{split} 
\| \chi v \|_{ L^2 ( \T^1) } & \leq C \left( \| F + 2 \chi' v \|_{ L^2
  (\T^1 ) } + \| G - \chi'' v \|_{L^2 ( \T^1 ) } + \| \chi v \|_{ L^2 (
  \T^1 )} \right) \\
& \leq 
C' \left( \| F  \|_{ L^2
  (\T^1 ) } + \| G  \|_{L^2 ( \T^1 ) } + \| v \|_{ L^2 (
  \T^1 )} \right) , 
\end{split} \]
that is, \eqref{eq:bsiam} holds for $ v $ satisfying \eqref{eq:Fl}.

Since $ \| f \|_{ H^{-1} } = \inf \{ \| F \|_{L^2 } + \| H\|_{L^2} : 
f = \partial_x F + H \} $, the estimate on $ \| u \|_{ L^2 ( \T^1 )}
$,
$ u ( x ) = e^{ 2 \pi i k x } v ( x ) $,  in \eqref{eq:lol} follows. 

To estimate $ \partial_x u $ we write
\[ \begin{split}  \| ( \partial_x + i k ) u \|^2_{L^2 ( \T^1 ) } &  = 
\langle  ( - ( \partial_x + i k ) ^2  - \tau ) u , u \rangle_{L^2 ( \T^1 ) } +
\tau \| u \|^2_{L^2 ( \T^1 ) } \\
& = 
 \langle f + g , u \rangle_{ L^2 ( \T^1 ) }  + \tau \| u \|_{L^2 ( \T^1)}^2
\\
&  \leq \| f \|_{H^{-1} ( \T^1 ) }   \| u \|_{H^1 ( \T^1)}  
 +  \| g \|_{L^2 ( \T^1 ) }  \| u \|_{L^2 ( \T^1)} 
 + \langle \tau \rangle \| u \|_{L^2 ( \T^1 )}^2 \\
& \leq \frac 12 \| ( \partial_x + i k )  u \|_{L^2 ( \T^1)} ^2 + C \| f \|_{  H^{-1} ( \T^1 ) }^2
+ C \| g \|_{L^2 ( \T^1)}^2 +  C \langle \tau \rangle \| u \|_{ L^2 ( \T^1 ) } ^2 . 
\end{split}
\]
Using the estimate for $ \| u \|_{ L^2 ( \T^1 ) } $ we obtain \eqref{eq:lol}.
\end{proof}

\begin{proof}[Proof of Proposition \ref{pL:1}] 
 With constant $ C_1 $ depending on $ \tau $ the estimate
\eqref{eq:s1d} follows from the unique continuation property 
for $ - \partial_x^2 + W $, $ W \in L^p $, $ p > 1$. 
As pointed out in \cite{JK}, this result in implicit in the paper of Schechter-Simon \cite{SS}

 To obtain the dependence of contants for large $ \langle \tau
\rangle $ we first observe that interpolation between the $ H^{-1} $ and $
L^2 $ estimates in Lemma \ref{l:old} shows that if 
$ ( -(\partial_x + i k ) ^2  - \tau) u= g + f $,
 then
\[ \begin{split} 
 \| u\|_{L^2}+ \langle\tau \rangle^{-\frac12}  \| \partial_x u\|_{L^2}
& \leq   C \langle \tau \rangle^{-\frac12}  \|g\|_{L^2}  
+ C  \langle \tau
\rangle^{\frac{s-1}2}  \| f  \|_{H^{-s}}  + C  \| u \|_{ L^2 (
  \omega ) } , 
\end{split} \]
for $ 0 \leq s \leq 1 $.
 As a consequence, if
$ ( - \partial_x^2 - \tau ) u = g - W u $, 
then 
\begin{equation}
\label{eq:spl1}  \begin{split}
  \| u \|_{ L^2  } &  \leq  C \langle \tau \rangle^{-\frac12} 
\| g\|_{L^2} 
 +  C
 \langle \tau \rangle^{ \frac{ s-1} 2}  \|Wu
\|_{H^{-s}} + C  \| u \|_{ L^2 ( \omega ) } . 
\end{split} \end{equation}

For $s<\frac12$,  $H^s( \T^1) \rightarrow L^{\frac 2
  {1-2s}}(\T^1)$ and hence, by duality, $ L^{\frac 2 {1+ 2s}}(\T^1)
\rightarrow H^{-s} ( \T^1)$. 
Choosing $s= \frac 1 {2p} < \frac 1 2$, and applying H\"older's
inequality we obtain 
\[ \begin{split} 
 \|Wu \|_{H^{-s}} &  \leq C \| W u \|_{ L^{\frac 2 { 1 + 2s } }  } 
\leq C \| W\|_{L^p} \|u\|_{L^{\frac{2 } { 1 - 2s } }  
}  \\
& \leq C  \| W \|_{ L^p  } \| u \|_{ H^s  }  \leq 
C' \| W \|_{ L^p  } \| u\|^{1-s}_{L^2} \left( \|  u\|_{L^2} + \| \partial_
x u \|_{L^2} \right)^{s} \\
& \leq C' \| W \|_{ L^p  } 
\left(   \langle \tau \rangle^{ ( 1 + \delta) \frac { s^2} { 2 ( 1 -s )} }   \|
  u\|_{L^2}+  \langle \tau \rangle^{-( 1+ \delta ) \frac s 2}  \| \partial_x
   u\|_{L^2} \right)  . 
\end{split} \]
Combining this with \eqref{eq:spl1} yields
\[ \begin{split} 
 \| u \|_{ L^2  } + \langle \tau \rangle^{-\frac12}
 \| \partial_x u\|_{L^2}  &  \leq  
 C \langle \tau \rangle^{-\frac12}  \|
 g\|_{L^2}    +  C  \| u \|_{ L^2 ( \omega ) }  
+ C_2  \langle \tau \rangle^{ \frac{s-1} 2} 
\langle \tau \rangle^{ ( 1 + \delta ) \frac{s^2} {2( 1 - s)} }  \| u
\|_{L^2} 
\\ 
& \ \ \ \ \ \ \ \ \ \ \ + 
C_3  \langle \tau \rangle^{\frac{s-1} 2  } 
\langle \tau \rangle^{ - ( 1+ \delta ) \frac s 2 }  \| u \|_{H^1} .
\end{split} \]
Since $ 0 < s < 1 $, taking $ \langle \tau \rangle $ large enough allows
us to absorb the last term on the right hand in the left
hand side.
Same is true for the third term since 
\[  \frac{ ( 1 + \delta ) s^2} { 2(1 - s) } + \frac{ s - 1} 2 = 
 \frac{  -1 + 2 s + \delta s^2 } { 1 -s } ,
\]
which  is negative for  $ 0 < s < \frac12 $ if we choose $ \delta $
small 
enough.
\end{proof}

\begin{proof}[Proof of Proposition~\ref{prop.1d}]
Let us now show how to pass from the estimate in
Proposition~\ref{pL:1} to an observability result. This was already
achieved in~\cite{BZ2} in a more general semiclassical setting. For
completeness we present a simple version of it here -- see \cite{Mi}.

 For $\chi \in C^\infty_0 ( \RR)$,  put $w= \chi(t)e^{itP} u_0$, which solves
\[ (i\partial_t + P ) w = i \chi'(t) e^{itP} u_0= v, \ \ P :=
- (\partial_x + i k ) ^2 + W ( x ) . \]
Taking Fourier transforms with respect to time, we get 
$$ (P- \tau) \widehat{w} ( \tau )= \widehat{v} ( \tau) .$$
Using the estimate in Proposition~\ref{pL:1},  we write
$$ \| \widehat{w} ( \tau)\|_{L^2 ( \T ) } \leq \frac{C} { 1+
  \sqrt{|\tau|}} \| \widehat{v} ( \tau)\|_{L^2 ( \T ) } + C  \|
\widehat{w} ( \tau) \|_{L^2 ( \omega) } . $$
Now, taking $L^2$ norm with respect to the $\tau$ variable, gives
$$ \| \widehat{w} ( \tau)\|_{L^2 ( \RR_\tau \times \T)  } \leq
\frac{C} { 1+ \sqrt{N}}\| \widehat{v} ( \tau)\|_{L^2 ( \RR_\tau \times
  \T) }+ C \| \widehat{w} ( \tau)\|_{L^2 ( \RR_\tau \times\omega) } +
\Bigl(\int_{|\tau \leq N} \| \widehat{v} ( \tau)\|^2_{L^2  ( \T ) }d\tau\Bigr)^{\frac12} .
$$ 
From this we notice that 
\begin{gather*}   \| \widehat{w} ( \tau)\|_{L^2 ( \RR_\tau \times \T) } = \| u_0\|
_{L^2 ( \T ) } \times \| \chi\|_{L^2(\RR ) }, \ \ \ \ \| \widehat{v} (
\tau)\|_{L^2 ( \RR_\tau \times \T) } = \| u_0\| _{L^2 ( \T )  } \times \| \chi' \|_{L^2(\RR)}, \\
\| \widehat{w} ( \tau)\|_{L^2 ( \RR_\tau \times \omega ) }= \|\chi(t)
e^{itP} u_0\|_{L^2(\RR_t \times \T)}.
\end{gather*}
From this we deduce that if 
$$ \frac{ C\|\chi'\|_{L^2}} { \| \chi\|_{L^2} ( 1 + \sqrt{ N}) } \leq
\frac 1 2,$$
then 
\begin{equation}
\label{eq:NN} \| u_0\|_{L^2_x}\leq C' \|\chi(t) e^{itP} u_0\|_{L^2(\RR_t \times
  \T_x)} + C' \Bigl(\int_{|\tau \leq N} \| \widehat{v} ( \tau)\|^2_{L^2_{\tau, x}}d\tau\Bigr)^{\frac12}  .
\end{equation}

To understand the last term on the right-hand side of we define 
Sobolev norms associated to $ P $. Let $ \{ \varphi_n \}_{n=1}^\infty $ be
an orthonormal basis of $ L^2 ( \T^1 ) $ consisting of eigenfuctions
of $ P$. We then put
\[   \| u \|_{ H_P^k }^2 := \sum_{ j=1}^\infty \langle \lambda_n
\rangle^{ 2k } |u_n |^2 , \ \ P \varphi_n = \lambda_n
\varphi_n , \ \  u_n :=  \langle u , \varphi_n \rangle . \]
In this notation $ w = \chi(t) \sum_{n} u_n e^{-it\lambda_n} \varphi_n $, 
and 
$$ \widehat{v} ( \tau)= \sum_n \widehat{\chi'} ( \tau - \lambda_n) u_n
\varphi_n . $$
Hence 
\[ \begin{split} 
\int_0^N \| \widehat{v } ( \tau)\|^2_{L^2_{ x}}d\tau & =
\sum_{ n=1}^\infty |u_n|^2 \int_0^N 
| ( \tau - \lambda_n ) \hat \chi ( \tau  - \lambda_n )|^2  d
\tau 
= \sum_{ n=1}^\infty | u_n|^2 \int_0^N  {\mathcal O} ( \langle \tau - \lambda_n
\rangle^{-\infty }  ) d \tau \\
& \leq 
C_{N , M} \sum_{ n=1}^\infty \langle \lambda_n \rangle^{-M} | u_n|^2 =
C_{N, M } \| u\|_{H_P^{-M} }^2 ,
\end{split}
\]
for any $ M $. Taking $ M = 2 $ and combining this with \eqref{eq:NN}
we obtain
\begin{equation}
\label{eq.1d}
\|u_0\|_{L^2( \T^1)} \leq C\|\chi(t) e^{itP} u_0\|_{L^2(\RR_t \times
  \omega)} +  C\| u_0\|_{H^{-2}_P( \T^1) }.
\end{equation}
 To complete the proof, it remains to eliminate the last term on the
 right hand side of~\eqref{eq.1d}. For this, 
we apply the now classical uniqueness-compactness argument of
Bardos-Lebeau-Rauch~\cite{BLR} (see also \cite[\S4]{BZ4}) or the direct argument presented in the
Appendix.
We note that both approaches rely on the unique continuation property
of $ - ( \partial_x + i k ) ^2 + W ( x ) $, $ W \in L^p ( \T^1) $, $ p > 1 $.
\end{proof}

For later use we also record the following approximation result:
\begin{prop}\label{cor.obs}
Assume that the sequence of potentials $W_j$ is converging to $W$ in
$L^p(\T^1)$, $p\geq 2$. Then there exist $ K_0>0$ such that for any $
k \in [0, 1 ) $ and $ u \in L^2 ( \T^1 ) $, and any $j\in \NN$,  
\begin{equation}
\label{eq:ob1d}
 \| u \|_{L^2 ( \T^1 ) }^2 \leq
{ K_0}   \int_0^T \| e^{ i t ((  \partial_x + i k )^2  - W_j ) } v \|_{ L^2 ( \omega )}^2 
dt . 
\end{equation}
\end{prop}
 \begin{proof} The proof follows from Proposition~\ref{prop.1d} by a
   simple perturbation argument. Put $ P = - ( \partial_x + ik)^2 + W
   $ and $ P_j = - ( \partial_x^2 + ik)^2 + W_j $. Then,
according to the Duhamel formula, we have 
 $$ e^{-it P  } v = e^{-itP_j  } v + \frac 1 i \int_0^t e^{-i(t-s) P_j } (W-W_j)e^{-isP } v ds,$$
 and consequently, according to~\eqref{eq.est3} we obtain
\[ \begin{split}
 \| e^{-itP } v -e^{-itP_j } v\|
 _{L^\infty([0,T]; L^2( \T^1))} & \leq C \|(W-W_j) e^{-isP } v \|_{L^1([0,T]; L^2( \T^1))} \\
&  \leq  C \sqrt{T} \|W-W_j\|_{L^2( \T^1)}  \|e^{-is P } v \|_{L^\infty(\T^1_x;  L^2( 0,T))}\\
&  \leq C \sqrt{T} \|W-W_j\|_{L^2( \T^1)}  \|v \|_{L^2( \T^1)}.
 \end{split} \]
 According to~\eqref{eq:ob1d-b} we have 
\[  \begin{split}
  \| v \|_{L^2 ( \T^1 ) }^2 & \leq
{ K_0}   \int_0^T \| e^{ - i t P } v \|_{ L^2 ( \omega )}^2\\
& \leq 2 K_0 \int_0^T \| e^{ -i t P_j } v \|_{ L^2 ( \omega )}^2 + 2 C^2 T \|W-W_j\|^2_{L^2( \T^1)}  \|v \|^2_{L^2( \T^1)}.
\end{split} \]
which implies~\eqref{eq:ob1d} if  $\|W-W_j \|_{L^2( \T^1)} $ is small enough. 
 \end{proof}

 \section{semiclassical observation estimates in dimension $2$}
\label{semi}

We revisit and refine the arguments of \cite{BZ4}. 
The key point in our analysis will be the following 
variant of \cite[Proposition 3.1]{BZ4}. The key difference is that
now the main constant is determined in terms of the geometry of the
problem and the potential $ V $.
  \begin{prop}\label{th:3}
Suppose that $ V_j \in \CI ( \T^2 ; \RR ) $ converge to  $ V $ in the $  L^2 (
\T^2 ) $ topology.  Let $\chi\in \CIc (  -1 , 1 )$ be equal to $1$ near $0$, and define
$$ \Pi_{h, \rho, j } (u_0)  := \chi \left( \frac{ h^2(- \Delta + V_j) -1}
  {\rho} \right) u_0\,,  \ \ \ \rho > 0 \,. $$
 Then for any non-empty open subset $ \Omega $ of $ \T^2 $ and $T>0$,
 there exists a constant $K>0$
such that for any $j$ there exist $\rho_j>0, h_{0,j}>0$  such that for any
$0<h<h_{0,j}$, $u_0\in L^2( \T^2)$, 
we have  
\begin{equation} \label{eq:semi}
 \|\Pi_{h, \rho_j, j } u_0\|_{L^2}^2 \leq K \int_{0}^T \| e^{-i t ( - \Delta +
   V_j )} \Pi_{h, \rho_j, j } u_0\|_{L^2 ( \Omega ) }^2 dt \,.
\end{equation} 
\end{prop}

In the proof
we argue by contradiction. We first observe that if the estimate~\eqref{eq:semi} is true for some $\rho>0$,
then is is true for all $0<\rho'<\rho$. As a consequence, if
~\eqref{eq:semi} were false then for any $j$,  there would exist sequences 
$$ h_{n, j} \longrightarrow 0, \quad \rho_{n, j} \longrightarrow 0, \quad u_{0,n,j}=
\Pi_{h_{n,j}, \rho_{n,j}, j} (v_{0,n,j}) \in L^2,  $$
$$  i \partial_t u_{n,j} ( t, z ) = ( - \Delta + V_j( z ) ) u_{n,j} ( t, z )
\,, \ \  u_{n,j} ( 0 , z ) = u_{0,n,j} ( z ) \,, $$
such that 
$$ 1= \| u_{0,n,j} \|^2_{L^2} , \qquad \int_{0}^T
\|u_{n,j} ( t , \bullet ) \|_{L^2(\Omega)}^2 dt \leq \frac 1 K\,. $$
Each sequence $n \mapsto u_{n, j} $ is bounded in $L^2_{\rm{loc}} ( \mathbb{R} \times
\T^2)$ and consequently, after possibly extracting a subsequence,
there exists  a semiclassical defect measure $\mu_j$ on $\R_t \times T^* (\T^2_z)$ such
that for any function $\varphi \in 
{\mathcal C}^0_0 ( \R_t)$ and any $a\in \CIc ( T^*\T^2_z)$, we have 
\begin{equation}
\label{eq:muj}
\begin{split} \langle \mu_j, \varphi(t) a(z, \zeta)\rangle &   = \lim_{n\rightarrow  \infty} \int_{\R_t \times \T^2} \varphi(t) 
( a(z, h_{n,j} D_z)  u_{n,j} ) ( t , z ) \overline u_{n,j} (t,z) dt dz \,.
\end{split} 
\end{equation}

\renewcommand\thefootnote{\ddag}%

Furthermore, standard arguments\footnote{see \cite{AM} for a review of recent 
results about measures used for the Schr\"odinger equation.} 
show that the measure $\mu_j$ satisfies
\begin{itemize}
\item 
\begin{equation}\label{eq.nonvan}
 \mu_j ( (t_0 , t_1 ) \times T^*\T^2_z) = t_1 - t_0  \,. 
\end{equation}
\item  The measure $\mu_j $ on $\R_t \times T^* (\T^2)$ is supported in the set 
$$ \Sigma := \{(t,z, \zeta) \in \R_t \times \T^2_z \times \R^2_\zeta
\; : \;  |\zeta|=1\}$$
and is  invariant  under the action of the geodesic flow:
\begin{equation}
\label{eq:invm} \xi \cdot \nabla_x ( \mu_j ) =0
\end{equation}
\item The mass of the measure on $\Omega$ is bounded away from $0$:
\begin{equation}\label{eq.nonvan-2}
\mu_j((0,T) \times T^*\Omega) \leq \frac 1 K.
\end{equation} 
\end{itemize}

We are going to show that a proper choice of the constant $K $ above
contradicts \eqref{eq.nonvan}.  
When no confusion is likely to occur we will drop the index $j$ for
conciseness.  

We start by decomposing $ \Sigma $ into to its rational and irrational 
parts. For that we identify $ \T^2 \simeq [ 0, A )_x \times [ 0, B)_y
$ where $ A , B \in \RR\setminus \{ 0 \}$, and define
$$ \Sigma_{\Q} : = \Sigma \cap \left\{ (t, z, \frac{ ( Ap,Bq)} {
    \sqrt{ A^2 p^2
      + B^^2 q^2}}); p, q \in \Z,  {\rm gcd}(p, q)  =1 \right\}.$$
The flow on $ \Sigma_{\Q } $ is periodic. Its complement is the set of irrational points:
$$ \Sigma_{\RR\setminus \Q} := \Sigma \setminus \Sigma_{\Q}$$ 
and it also invariant under the flow.

\subsection{The irrational directions}
\label{irr}
For simplicity we assume here that $ A = B = 2\pi $, that is 
$ \T^2 = \T^1 \times T^1 $,  as the argument
is the same as in the general case.

Let us first define $\mu_{\RR\setminus \Q}$ to be the restriction of 
the measure $\mu$ to $ \Sigma_{ \RR \setminus \Q}$. Since $ \mu $ is 
invariant, for any open set $\Omega \subset T^2$, and any $s\in \RR$, 
$$ \mu_{\RR\setminus \Q}((t_1, t_2) \times \Omega \times \R^2) = \mu_{\RR\setminus \Q}  ((t_1, t_2) \times  \Phi_s(\Omega \times \R^2))
$$
where the flow $\Phi_s$ is defined by  $ \Phi_s (z , \zeta)= (z+ s\zeta,
\zeta)$. 
As a consequence, we obtain 
\[ \begin{split} 
 \mu_{\RR\setminus \Q} ((t_1, t_2) \times \Omega \times \R^2) & = \frac 1 T \int_0 ^T \mu_{\RR\setminus \Q}  ((t_1, t_2) \times  \Phi_s(\Omega \times \R^2)) \\
&  = \int \bbbone_{t\in (t_1, t_2)} \times \frac 1 T \int_0 ^T \bbbone_{(z, \zeta)
  \in \Phi_s(\Omega\times \RR^2)} ds d\mu_{\RR\setminus \Q} . 
 \end{split}. \]
The equidistribution theorem shows that for any $(z, \zeta) $ in the support of  $\mu_{\RR\setminus \Q}$, 
 $$ \lim_{ T \to \infty } \frac 1 T \int_0 ^T \bbbone_{(z, \zeta) \in
   \Phi_{s} ( \Omega\times \RR^2) } ds = \frac {{\vol}( \Omega)} { {\vol} ( \T^2)}. $$
Hence the dominated convergence theorem and \eqref{eq.nonvan} show that
 \begin{equation}\label{eq.contrainte1}
   \mu_{\RR\setminus \Q} ((t_1, t_2) \times \Omega \times \R^2)= \frac {{\vol}( \Omega)} { {\vol} ( \T^2)} \mu_{\RR\setminus \Q} ((t_1, t_2) \times \T^2 \times \R^2).
   \end{equation}

 \subsection{Dense rational directions} 
\label{dr}

 We now consider the restriction of the measure $\mu$ on the set of
 rational directions, $\Sigma_{\Q}$.  We first consider the case of $
 p/q $ for which $ p^2 + q^2 $ is large (we again assume that $ A = B
 = 1 $ as the general argument is the same). In some sense that
 corresponds to being close to the irrational case.

 \begin{lem} \label{lem.rat}
 For any open set $\Omega$, there exists $N\in \N, \delta >0$ such
 that for any  $(p, q) \in \Z^2$, $ {\rm gcd}(p, q)  =1$, $\sqrt{p^2 + q^2}\geq N$, 
 $$ \liminf_{T\rightarrow + \infty} \frac 1 T \int_0^T \bbbone_{(z,
   \zeta) \in \Phi_{s}(\Omega\times \RR^2)} ds \geq \delta ,  \ \ \ 
\zeta = \frac{ ( p ,q ) } { \sqrt { p^2  + q^2 } } . $$
 \end{lem}
\begin{proof}
For any $z_0= (x_0, y_0)\in \Omega$ choose $ N >  4 \pi/ \epsilon $ 
where $B(z_0, 2 \epsilon) \subset \Omega$. Assume that $p\geq N/2 > 2
\pi / \epsilon $ and that $ p \geq q $ (the case of $ q \leq p $ is similar).
Put  
\[ s_k := \frac{ \sqrt{ p^2 + q^2} } p (  2k \pi - x_0 ) , \  \ \ \ k
= 0, \cdots, p-1 .\]  
Since $ p $ and $  q $ are co-prime  $q$ is a generator of the group
$\Z /p\Z$. Consequently, the points 
\[ Y_k = \frac {s_k} {\sqrt{p^2+q^2}} q  - y_0  \in   \T^1, \]
are at distance exactly $2\pi /p$ from each other. (Here, and below,
addition on $ \T^1$ is meant $\mod 2 \pi \Z $.)
We conclude that for any $ z \in \T^1 $ there exists 
\[ J_z \subset \{ 0, \dots, p-1\}, \ \ \ \ \ | J_z |  = [ \textstyle{\frac {\epsilon
  p} {\pi}} ] , \ \ \  \text{ such that for $ k \in J_z , \  |y + Y_{k} - y_0| \leq
\epsilon. $} \]
Since the flow is given by 
 $$ \Phi_{-s} \left ( (x,y), \frac{ (p,q)} { \sqrt{p^2 + q^2}} \right)
 = \left((x,y)- \frac s {\sqrt{p^2+q^2}} (p,q), \frac{ (p,q)} {
     \sqrt{p^2 + q^2}}\right),$$
for any $k\in J$, 
$ \Phi_{-s_{k}} \left( z ,{ (p,q)}/ { \sqrt{p^2 + q^2}} \right)
\in B( z_0, \epsilon)\times \RR^2 $. 
Since $  {2\pi} / p < \epsilon $, we also obtain that for $|s-s_{k}|<\epsilon$  
 $$\Phi_{-s} \left( z ,\frac{ (p,q)} { \sqrt{p^2 + q^2}} \right) \in B( z_0, 2
 \epsilon )\times \RR^2 \subset \Omega \times \RR^2 .$$ 
Hence, using the assumption that $ q \leq p $, 
$$  \int _0 ^{ 2 \pi \sqrt { p^2 + q^2 } } \bbbone_{ \Phi_{-s} (
   z, \zeta ) \in \Omega \times
   \RR^2} ds \geq  
[ \textstyle{ \frac {\epsilon p} {\pi} }  ]\epsilon 
> 2 \pi \sqrt { p^2 + q^2 } \delta 
, \ \ \ 
\zeta = \frac{ ( p ,q ) } { \sqrt { p^2  + q^2 } } , $$ 
for some $ \delta > 0 $. 
Since the evolution of $ ( z, \zeta ) $  is periodic with period $ 2 \pi \sqrt { p^2 + q^2}
$, the lemma follows. 
\end{proof} 

Let us now fix $N$ as in Lemma~\ref{lem.rat} and Let $\mu_{\Q ,N}$ be
the restriction of $\mu_{\Q}$ 
to rational directions satisfying 
$\sqrt{p^2 + q^2} \geq N$. As in the study of the irrational
directions, Lemma~\ref{lem.rat} and Fatou's  Lemma imply
 \begin{equation}\label{eq.contrainte2}
  \mu_{\Q,N} ((t_1, t_2) \times \Omega \times \R^2)\geq \delta \mu_{\Q,N} ((t_1, t_2) \times \T^2 \times \R^2).
  \end{equation}

  \subsection{Isolated rational directions}
This section is  closest to the arguments of \cite[\S3]{BZ4}. We
allow here existence of points in $ \Sigma_{\Q} $ whose evolution
misses $ \Omega $ altogether. The contradiction is derived from that
assumption. It is now important to keep $ A$ and $ B $ arbitrary,
$  \T^2 = \RR^2 / A \ZZ \times B \ZZ $.
The constraints on the constant ${K}$ will not be only geometric as in
\S\S \ref{irr},\ref{dr},  but will also involve the limit potential
$V$. Hence we return to the notation of \eqref{eq:muj} 
and keep the index $j$.

\begin{figure}[ht]
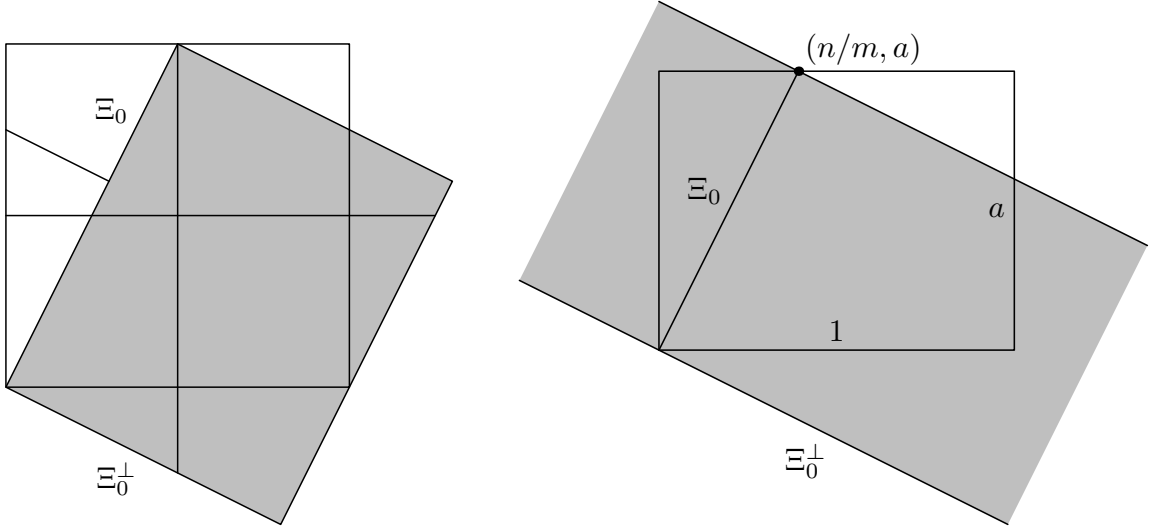

\includegraphics[width=6cm]{m11.mps}\qquad\includegraphics[height=7cm]{m12.mps}
\caption{On the left, a rectangle, $ R $, covering a rational torus $ \TT^2                      
  $. In that case we obtain a periodic solution on $ R $. On the
  right, the irrational case: the strip with sides $ m \Xi_0                                     
  \times   \RR \Xi_0^\perp $, $ \Xi_0 = ( n/m, a ) $ (not normalized
  to have norm one),
  also covers the torus $ [ 0 , 1 ] \times [ 0 , a ] $. Periodic
  functions are pulled back to functions satisfying
  \eqref{eq:per}. This figure is borrowed from \cite{BZ4}.}
\label{f:3}
\end{figure}

We consider the restriction of the measure $\mu$ to
any of the finitely many isolated rational directions:
\begin{equation}
\label{eq:Xi0}
\Xi_0 = \frac{ (Ap,B q)} {\sqrt{A p^2  + B q^2}},  \ \  \sqrt{p^2 + q^2} \leq N 
\end{equation}

We first recall the following simple result \cite[Lemma 2.7]{BZ4}
(see Fig.~\ref{f:3} for an illustration).
\begin{lem}
\label{l:si}
Suppose that $ \Xi_0 $ is given by \eqref{eq:Xi0} and
\begin{equation}
\label{eq:F} \; : \; ( x, y ) \longmapsto  z = F ( x, y )  = x
\Xi_0^\perp  + y \Xi_0 \,, \ \ 
\ \Xi_0^\perp  = \frac1 { \sqrt { n^2 A^2 + m^2 B^2 } } ( - m B , n
A  ) \,. 
\end{equation}
If 
 $ u = u ( z ) $ is perodic with respect to $ A \ZZ \times B  \ZZ$
then 
\begin{equation}
\label{eq:per}
 F^* u ( x + k a , y + \ell b ) = F^* u ( x , y - k \gamma ) \,, \ \ k,
 \ell \in \ZZ \,,  \ \ ( x, y ) \in \RR^2 \,, 
\end{equation}
where, for any fixed $ p , q \in \ZZ $, 
\[ a =   \frac {( q n - p m ) A B } { \sqrt {  n^2 A^2 + m^2 B ^2 } }    \,,
\ \      b =  \sqrt {  n^2 A^2  + m^2 B^2 }  \,, \ \  \gamma = -
\frac{ pn A^2 + q m B^2  }{  \sqrt { n^2 A^2 + m^2 B^2 } }  \,. \]
When $ B/A = r/s \in \Q $ then 
\[  F^* u ( x + k \widetilde a , y + \ell b ) = F^* u ( x, y ) \,, \ \  \ \ k,
 \ell \in \ZZ \,,  \ \ ( x, y ) \in \RR^2 \,, \]
for $ \widetilde a = ( n^2 s^2 + m^2 r^2 )  a  $.
\end{lem}

We now identify $ u_{n,j} $ with $ F^* u_{n,j} $, and
consider the Schr\"odinger equation on the strip
$ R = \RR_x \times [0,b]_y $ (or the rectangle 
$ R [ 0 , a]_x \times [0,b]_y $ in the case when $ A/B \in \Q $). 
In this coordinate system $ \Xi_0 = ( 0, 1) $.

Choosing a function $\chi\in C^\infty_0 ( \R^2)$ equal to
$1$ near $(0,0)$ we define, for $\epsilon >0$, 
\[ \chi_\epsilon := 
\chi(((\eta, \zeta)- (0,1) )/ {\epsilon} ), \ \  \eta, \zeta \in \RR
, \]
and 
$$ u_{n, j, \epsilon} ( x, y ) = \chi_\epsilon (h_{n,j} D_x)   u_{n,j} \,.  $$ 
We denote by $\mu_{j,\epsilon}$, the semiclassical
measure of the sequence $({u}_{n,j,\epsilon})_{n\in \NN}$ ($j, \epsilon$ are parameters).
Since 
$ \mu_{j,\epsilon} = (\chi_\epsilon  (\zeta))^2\mu_j $, 
(where we skipped the pull-back by $ F $
we have
\begin{equation}\label{eq.local}
\lim_{ \epsilon \to 0 +}  \mu_{j,\epsilon } =  \mu_{j} |_{ \{
   (t,z, \zeta ) \, : \,
\zeta =  (  0 , 1 )  \} }  \end{equation}
We now recall the following normal-form result given in \cite[Proposition
2.3]{BZ4} and \cite[Corollary 2.4]{BZ4}:
\begin{prop}
\label{p:2} 
Suppose that $ F : \RR^2 \rightarrow \RR^2 $ is given by
\eqref{eq:F}  and that $ V \in \CI ( \RR^2 ) $ is periodic
with respect to $ A \ZZ \times B \ZZ $. Let $ a , b $ and $                                      
\gamma $ be as in \eqref{eq:per}.

Let $\chi \in \CIc ( \mathbb{R}^2) $ be equal to $ 0 $ in a
neighbourhood of $ \eta = 0 $. Suppose 
that $ V_j ( x, y ) \in \CI ( \TT^1 \times \TT^1 ) $. Then there
exist operators
\[   Q_j( x, y , hD_y ) \in \CI ( \RR ) \otimes \Psi^{0} ( \RR )  \,, \ \ R_j ( x, y ,
hD_x, hD_y  )
\in \Psi^{0 } ( \RR^2 )  \,, \]
such that 
$  (F^{-1})^*  Q F^* $ and $  (F^{-1})^*  R F^* $ preserve
$ A \ZZ \times B \ZZ $ periodicity, and
\begin{equation}
\label{eq.normale}
\begin{split}
& ({\Id} + h Q_j) \left(D^2_y + F^* V_j ( x, y ) \right)\chi( hD_x,  hD_y)\\
& \ \ \ \ \ \ \ \ = (D_y^2 + W_j ( x)) ({\Id} + hQ_j) \chi( hD_x,
hD_y) + hR_j  
  \,, \end{split}
\end{equation}
where 
$  W_j ( x) = \frac 1 { b} \int_0^b F^* V_j ( x , y ) dy
$ satisfies $ W_j ( x + a ) = W_j ( x ) $.

Moreover, there exist operators $ P_j =  P_j   ( x, y, h D_x, h D_y ) \in \Psi^0 (
\RR^2 ) $ such that (with properties as above)
\begin{align}
& (\Id +h Q_j) \left(  D_x ^2 + D^2_y + F^* V_j ( x, y ) \right)\chi( hD_x, hD_y)\label{eq:22} \\
& \quad = \Bigl(\bigl(D_x^2 + D_y^2 + W_j ( x)\bigr) ( \Id+ hQ_j)
+P_j \Bigr) \chi(h D_x,  hD_y)  + h R_j ,\notag
\end{align}
\begin{equation}\label{eq.W}
P_j (x,y, x, \eta) = \frac 2  i \xi \partial_x q_j (x, y, \eta) \widetilde
\chi _\epsilon ( \xi, \eta), \ \ \ q_j = \sigma ( Q_j ) , 
\end{equation}
where $ \widetilde \chi \in \CIc ( \RR^2 ) $ is equal to one on the
support of $ \chi $.
\end{prop}

Using Proposition \ref{p:2} we define
$$ v_{n,j,\epsilon}  = \Bigl(1+hQ_j \Bigr)  u_{n, j , \epsilon}  \,,
\ \ h = h_{ n,j} . $$
Since the operator $Q_j$ is bounded on $L^2$,  
the semiclassical defect measures associated to $v_{n, j, \epsilon}$ and
${u}_{n,j,\epsilon}$ are equal.
We now consider the time dependent Schr\"odinger equation satisfied 
by $ v_{n,j,\epsilon} $. With 
\begin{equation}\begin{gathered}
 Q_{n,j} := Q_j ( x, y , h_{n,j} D_y ) \,, \ \ R_{n,j} := R ( x, y , h_{n,j} D_x, h_{n,j}D_y ) \,, \\
 P_{n,j} := P_j ( x, y , h_{n,j} D_x , h_{n,j}D_y ) \,, 
 \end{gathered}
 \end{equation}
given in Proposition~\ref{p:2} and $ \chi_{n,j, \epsilon} := \chi ( h_{n,j} D_z ) $, we have 
\begin{equation}\label{eq:v}
\begin{split}
 (i \partial_t + \Delta -W_j(x) ) v_{n,j} & = ({\Id} + h_{n,j} Q_{n,j} ) (i \partial_t +
 \Delta - V_j(x,y) ) \chi_{n,j,\epsilon}  u_{n,j}  \\
 & \qquad -P_{n,j}   \chi_{n,j,\epsilon}    {u}_{n,j} - h_{n,j} R_{n,j, \epsilon}u_{n,j}  
 \\
&  = -P_{n,j} \chi_{n,j,\epsilon}  {u}_{n,j} + [V, \chi_{n,j,\epsilon}
] u_{n,j} + o  _{L^2} (1) \\
&  = -P_{n,j} \chi_{n,j,\epsilon}  {u}_{n,j}+o _{L^2_{x,y}} (1) 
\end{split} 
\end{equation}
We also recall that according to~\eqref{eq.W}, on the support of $\mu_{j,\epsilon}$,   the symbol of the operator $W$ is smaller than $C\epsilon$. 
 This implies that
 \begin{equation}\label{eq:vbis}
 (i \partial_t + \Delta -W_j(x) ) v_{n,j,\epsilon} = f_{n,j, \epsilon}
 \end{equation}
 with 
\begin{equation}\label{eq.loc}
\limsup_{n\rightarrow + \infty} \| f_{n,j,\epsilon}\|^2_{L^2 ( [0,T]\times \T^2)}= \langle \mu_{j, \epsilon}, |P_{n,j}|^2\rangle  \leq C_j \epsilon^2
\end{equation}
 
The following simple observation
$$ e^{it( \partial_y^2 + \partial_x^2- W_j (x))} = e^{it \partial_y^2 } e^{it( \partial_x^2- W_j (x))} .$$
shows that we can write
$$ v_{n,j, \epsilon } ( t, x , y ) = \sum_{ k \in \ZZ } e^{-i(tk^2 + ky)}
v_{n,j,\epsilon, k} (t,x), \qquad f_{n,j, \epsilon}  ( t , x , y ) = \sum_{ k \in \ZZ
}  e^{-i ky} f_{n,j,\epsilon, k} (t,x) \,, $$
where the coefficients satisfy a Floquet condition (see \cite[Proof of
Proposition 2.2]{BZ4})
\[ \begin{split} &  v_{n,j,\epsilon,k} ( t, x + a ) = e^{ 2\pi i \gamma k/b }
v_{n,j,\epsilon,k} ( t , x) = e^{ 2\pi i \gamma_k } v_{n,j,\epsilon,k}
( t, x ) , \\ 
& f_{n,j,\epsilon,k} ( t, x + a ) = e^{ 2\pi i \gamma_k }
f_{n,j,\epsilon,k} ( t, x ) , \ \ \gamma_k := \gamma k /b = [\gamma k
/b ] \in [0, 1 ) . \end{split} \]




Since $ W_j ( x + a ) = W_j ( x ) $ and 
\[ \begin{split}   \| W - W_j \|_{ L^2 ( [ 0 , a ]_x)  }^2 & = \int_0^a\left( \frac 1 b
  \int_0^b \int ( F^*V ( x, y ) - F^*V_j ( x , y ) ) dy \right)^2 dx
\\
& \leq \| F^* ( V - V_j ) \|_{ L^2 ([ 0 , a ]_x \times [ 0 , b]_y )
}^2 \\
&  \leq C_{\Xi_0} \| V - V_j \|_{L^2 ( \T^2 ) } \longrightarrow 0 ,  \
\ j \longrightarrow \infty, 
\end{split} \]
we can 
apply the one dimensional Proposition~\ref{cor.obs}. 
For that we fix a domain $\omega \subset [0, a]_x $ such that for any $x\in
\overline{\omega}$, the line  $\{x\} \times [0, b]_y $,  encounters $
\Omega$.  The estiaate \eqref{eq:ob1d} 
gives the
following non-geometric estimate; it is here where the depenence on
the norm of the potential enters:
\[ \begin{split} 
 \| v_{n,j,\epsilon, k}\|^2_{L^\infty([0,T]; L^2( [0,a]_x)) } & \leq 2\|v_{n,j,\epsilon, k}\mid_{t=0} \|_{ L^2( [0,a]_x) } + 2\|  f_{n,j,\epsilon, k} \|^2_{L^1([0,T]; L^2([0,a]_x))} \\
& \leq    K_0 \int_0 ^T \| e^{it( \partial_x^2- W_j (x))} v_{n,j,\epsilon, k}\mid_{t=0}\| ^2_{L^2(\omega)} + C \|  f_{n,j,\epsilon, k} \|^2_{L^2([0,T]\times [0,a]_x)}\\
& \leq    K_0 \int_0 ^T \|  v_{n,j,\epsilon, k}\| ^2_{L^2(\omega)} + C
\|  f_{n,j,\epsilon, k} \|^2_{L^2([0,T]\times [0,a]_x) }. 
\end{split}
\] 
Summing over $k\in \Z$ gives
$$\| v_{n,j,\epsilon}\|^2_{L^\infty([0,T]; L^2( [0, a ] \times [0, b]_y) } \leq K_0 \int_0 ^T \|  v_{n,j,\epsilon}\mid_{t=0}\| ^2_{L^2(\omega)} + C \|  f_{n,j,\epsilon} \|^2_{L^2([0,T]\times [0,a]_x)} 
$$
Taking first the limit $n\rightarrow + \infty$, we obtain, according to~\eqref{eq.loc} 
\[ \mu_{j, \epsilon} (( 0,T) \times ([0,a ] \times [0,b]_y)  \times \R^2)\leq K_0 \mu_{j,\epsilon}( ( 0,T) \times \omega \times [0,b]_y \times \R^2 ) + C _{j} \epsilon .
\]
Then taking the limit $\epsilon \rightarrow 0$, we conclude that,
according to~\eqref{eq.local}, 
\begin{equation}
\label{eq:limep} \mu_j ( (0,T) \times ([0,a]_x\times[0,b]_y) \times \{ (0,1)\}\leq K_0  \mu_j ( (0,T) \times {\omega} \times [0,b]_y \times \{ (0,1)\})
\end{equation}
Since vertical line over $\overline \omega$ encounters the open set $\Omega$,  we
have 
\[ \min_{ x \in \overline \omega } \int_{ \Omega \cap (\{ x \} \times    [0,b]_y } dy > \delta_0 > 0 . \]
This and the invariance of the measure under the flow (which now is
just the translation in the $ y $ direction) imply that
$$\mu_j ( (0,T) \times {\omega} \times [0,b]_y \times \{ (0,1)\})\leq
 \delta_0 \mu_j ( (0,T) \times \Omega \times \{ (0,1)\}).
$$
Combining this with \eqref{eq:limep}  we obtain that there exists a
constant $K_{(0,1)}$, independent of $j$, such that 
$$\mu_j ( (0,T) \times ([0,a]_x\times[0,b]_y) \times \{ (0,1)\})\leq K_{(0,1)} \mu_j (
(0,T) \times \Omega \times \{ (0,1)\}) .$$
Returning to an arbitrary rational direction, 
$$ \zeta_{p,q}= \frac{ (p,q)} {\sqrt{A p^2 + B q^2}},  \ \ \sqrt{ p^2 + q^2 } \leq N,$$
 we obtain that there exists a constant $K_{p,q}$ such that 
\begin{equation}\label{eq.contrainte3}
\mu_j ( (0,T) \times \T^2 \times \zeta_{p,q}) \leq K_{p,q} \mu_j ( (0,T) \times \Omega \times \Xi_{p,q})
\end{equation}

\subsection{Conclusion of the proof of Proposition \ref{th:3}}
If the constant $ K $ in the statement of the proposition is chosen so
that, with $ \delta $ in \eqref{eq.contrainte2},
\[ \frac K T > \max \left( \frac{{\vol} ( \T^2)}{ { \vol} ( \Omega)},
  \frac 1 \delta, \max_{ \sqrt {p^2 + q^2} \leq N }  K_{p,q} \right) , 
\]
then, according to~\eqref{eq.contrainte1},  \eqref{eq.contrainte2}
and~\eqref{eq.contrainte1}, we must have 
$$ \mu(( 0,T) \times \T^2 \times \R^2) < T,
$$ which contradicts ~\eqref{eq.nonvan} and completes the proof of
Proposition \ref{th:3}.

\section{From smooth to rough potentials}
\label{frsm}

Proposition \ref{th:3} was proved under the assumptions that $ V_j \in
\CI ( \T^2 ) $ converge to $ V \in L^2 ( \T^2 )$. To pass to $ L^2 $
potentials
we will now use the  results on \S \ref{cd2}.

\subsection{Classical observation estimate for smooth potentials}
The first proposition is the analogue of \cite[Proposition 4.1]{BZ4}
but with constants described by Proposition \ref{th:3}.

\begin{prop}\label{th.1}
Suppose that $ V_j \in \CI ( \T^2 ; \RR ) $ converge to  $ V $ in the $  L^2 (
\T^2 ) $ topology. 
 Then for any non-empty open subset $ \Omega $ of $ \T^2 $ and $T>0$, there exists $C>0$
such that for any $j\in \NN$ there exists $C_j$ such that for any
$u_0\in L^2( \T^2)$,  we have
\begin{equation}
\label{eq:th.1} \|u_0 \|_{L^2 ( \T^2)} \leq C \| e^{it ( \Delta - V_j
  )} u_0 \|_{L^2([0,T] \times \Omega)} + C_j\|u_0 \|_{H^{-1} ( \T^2)} , 
\end{equation}
\end{prop}
\begin{proof}
To obtain the estimate \eqref{eq:th.1}  from Proposition~\ref{th:3}, we apply pseudodifferential calculus in the time variable. This was already performed in~\cite{BZ4}, but since we need a precise dependence on the constants we recall the argument.
Consider a $j$-dependent partition of unity
\begin{gather*}  1= \varphi_{0,j} (r)^2 + \sum_{k=1}^{\infty}
\varphi_{k,j} ( r )^2, \quad  \varphi_{k,j} ( r ) :=   \varphi(R_j ^{-k} |r| ), \enspace R > 1, \\
\varphi \in \CIc ( ( R_j^{-1}, R_j ) ; [0, 1]  ), \quad  (
R_j^{-1}, R_j) \subset \{ r \; : \;  \chi( r / \rho_j)  \geq
\textstyle{\frac12} \},
\end{gather*}
where $ \chi $  and $ \rho_j $ come from
Proposition~\ref{th:3}. 
Then,  we decompose $u_0$
dyadically:
\[
\|u_0 \|^2_{L^2} = \sum_{k=0}^\infty  \|\varphi_{k,j}( P_{V_j} ) u_0
\|_{L^2}^2. \quad  P_{V_j} := - \Delta + V_j.
\]
Let $\psi \in \CIc ( ( 0, T ) ; [ 0, 1 ] ) $ satisfy  $\psi ( t ) > 1/2$,
on $ T/3 < t < 2T / 3 $.
We first observe (using the time translation invariance of Schr\"odinger equation) that in
Proposition~\ref{th:3} we have actually proved  that
\begin{equation}\label{eq:piece}
\|\Pi_{h, \rho_j,j} u_0\|_{L^2}^2 \leq K \int_\RR \psi( t )^2 \| \mathrm{e}^{-\mathrm{i} t ( - \Delta + V_j )} \Pi_{h, \rho_j,j} u_0\|_{L^2 ( \Omega ) }^2 dt, \quad 0 < h < h_0,
\end{equation}
which is the version we will use.

Taking $K_j$ large enough so that $R^{-K_j} \leq h_{0,j}$, where $
h_0 $ is as in Proposition~\ref{th:3}, we apply
\eqref{eq:piece} to the dyadic pieces:
\begin{align*}
\|u_0 \|^2_{L^2} & =\sum_{ k \in \ZZ }  \|\varphi_{k,j}( P_{V_j})u_0\|_{L^2}^2 \\
&\leq \sum_{k=0}^{K_j}  \|\varphi_{k,j}( P_{V_j} )u_0\|_{L^2}^2 +C
\sum_{k=K_j+1}^\infty  \int_{0}^T \psi ( t ) ^2 \|  \varphi_{k,j}(  P_{V_j})\,
\mathrm{e}^{-\mathrm{i}t P_{V_j} } u_0\|_{L^2( \Omega) } ^2 dt \\
& = \sum_{k=0}^{K_j}  \|\varphi_{k,j}( P_{V_j} )u_0\|_{L^2}^2 + C
\sum_{k=K_j+1}^\infty  \int_\RR \| \psi ( t ) \varphi_{k,j}(  P_{V_j})
\,\mathrm{e}^{-\mathrm{i}t P_{V_j} } u_0\|_{L^2( \Omega) } ^2 dt.
\end{align*}
Using the equation we can replace $ \varphi ( P_{V_j} ) $ by $
\varphi ( D_t ), $ which meant that we did not change the
domain of $ z $ integration. We need to consider the
commutator of $  \psi \in \CIc ( ( 0,  T) ) $ and $
\varphi_{k,j} ( D_t )\,{=}\,\varphi ( R^{-j}  D_t ) $. If  $
\widetilde \psi \in \CIc ((0,T))$  is equal to $ 1 $ on $
\supp  \psi$ then the semiclassical pseudo-differential
calculus with $ h = R_j^{-k}$ (see for instance
\cite[Chapter~4]{EZB}) gives
\begin{equation}\label{eq:wideps}
\psi ( t )  \varphi_{k,j} ( D_t )  =    \psi ( t )  \varphi_{k,j}
(D_t )   \widetilde  \psi ( t ) +  E_j ( t, D_t),\quad
\partial^\alpha E_j  = {\mathcal O}   ( \langle t \rangle^{-N } \langle
\tau \rangle^{-N}
R_j^{- N k } ),
\end{equation}
for all $ N$ and uniformly in $ k $.

The errors obtained from $ E_k $ can be absorbed into the $
\| u_0 \|_{ H^{-2}  ( \TT^2 ) } $ term on the right-hand
side (with a constant depending on $j$). Hence we obtain
\begin{align*}
\|u_0 \|^2_{L^2} &\leq C_j\|u_0 \|_{H^{-2} ( \T^2)} ^2+ C \sum_{k=0}^\infty \int_{0}^T
 \|   \psi (  t)  \varphi_{k,j}(  D_t  )\, \mathrm{e}^{- \mathrm{i}t P_{V_j} }u_0 \|^2_{L^2( \Omega)} dt\\
& \leq\widetilde C_j\|u_0 \|_{H^{-2} ( \T^2)} ^2+ K   \sum_{k=0}^\infty
\langle   \varphi_{k,j}(  D_t  )^2  \widetilde \psi ( t ) \,\mathrm{e}^{-\mathrm{i}t P_{V_j} }u_0,  \widetilde \psi( t )\,
\mathrm{e}^{-\mathrm{i}t P_{V_j}}u_0,  \rangle_{L^2 ( \RR_t \times \Omega)} \\
& = \widetilde  C_j\|u_0 \|_{H^{-2} ( \T^2)} ^2 + K \int_\RR  \| \widetilde \psi ( t ) \,\mathrm{e}^{-\mathrm{i}t P_V } u_0\|^2_{L^2( \Omega)} dt  \\
& \leq \widetilde C_j\|u_0 \|_{H^{-2} ( \T^2)} ^2 + K \int_{0}^T  \|\, \mathrm{e}^{-\mathrm{i}t P_V } u_0\|^2_{L^2( \Omega)} dt,
\end{align*}
where the last inequality is the statement of the
proposition.
\end{proof}

\subsection{Proof of Theorem \ref{th.2}.} 
We can now deduce Theorem \ref{th.2} from Proposition \ref{th.1}.
For that  we consider a sequence 
     $V_j$ of smooth potentials  converging to $V$ in $L^2( \T^2)$ (to construct such sequence, consider the Littlewood-Paley cut-off $V_j= \chi( 2^{-2j} \Delta) V$,$ \chi \in C^\infty _0 ( \R)$ equal to $1$ near $0$). We now have according to Proposition~\ref{th.1}
  $$ \|u_0 \|_{L^2 ( \T^2)} \leq C \| e^{it ( \Delta - V_j)} u_0 \|_{L^2([0,T] \times \Omega)} + D_j\|u_0 \|_{H^{-2} ( \T^2)}.
$$ 
On the other hand, according to~\eqref{eq.diff}, we have 
$$\| e^{it ( \Delta - V_j)} u_0 \|_{L^2([0,T] \times \Omega)}\leq  \| e^{it ( \Delta - V)} u_0 \|_{L^2([0,T] \times \Omega)}+C\| V-V_j\|_{L^2} \| u_0 \|_{L^2(\T^2_x)},
$$
and consequently, we deduce 
$$ \|u_0 \|_{L^2 ( \T^2)} \leq C \| e^{it ( \Delta - V)} u_0 \|_{L^2([0,T] \times \Omega)} + C\| V-V_j\|_{L^2}\|u_0\|_{L^2( \T^2_x)} +  D_j\|u_0 \|_{H^{-1} ( \T^2)},
$$ and consequently, taking $j$ large enough so that $C \|
V-V_j\|_{L^2}\leq \frac 1 2$, we conclude that 
$$ \|u_0 \|_{L^2 ( \T^2)} \leq 2C \| e^{it ( \Delta - V)} u_0 \|_{L^2([0,T] \times \Omega)} + 2  D_j\|u_0 \|_{H^{-1} ( \T^2)}.
$$
It remains to eliminate the last term in the right-hand side of this
inequality. For this we use again 
classical uniqueness-compactness argument of
Bardos-Lebeau-Rauch~\cite{BLR} (see also \cite[\S4]{BZ4}) or the direct argument presented in the
Appendix. The needed unique continuation results for $ L^2 $
potentials in $ \RR^2 $ follows, as it did in \S \ref{cd1}  from the
results of \cite{SS}.

\newpage

\appendix
\section{A quantitative version of the uniqueness-compactness
  argument}

We present an abstract result which eliminates the low-frequency
contributions in observability estimates. 

Let $ P $ be an unbounded self-adjoint operator on a Hilbert spaces $
\cH $. We assume that the spectrum of $ P $ is discrete:
\[  P \varphi_n = \lambda_n \varphi_n , \ \ \lambda_1
\leq \lambda_2 \leq \cdots , \ \ \ \lambda_n \geq n^\delta/C_0 , \ \
\delta > 0 , \]
where $ \{ \varphi \}_{n=1}^\infty $ form an orthonormal basis of $ \cH
$.  

We define $ P$-based Sobolev spaces using the norms
\begin{equation}
\label{eq:defHP}   \| \varphi \|_{ \cH^s_P }^2 := \sum_{n=1}^\infty \langle \lambda_n
\rangle^{2s} | \langle \varphi , \varphi_n \rangle  |^2. \end{equation}
The Schr\"odinger group for $ P $ is the following unitary operator
on $ \cH$:
\[  U ( t ) \varphi = \exp ( - i t P ) \varphi = 
\sum_{n=1}^\infty \langle \varphi , \varphi_n \rangle 
e^{ - i t \lambda_n } \varphi_n . \]

We have the following general result:

\begin{thm}
\label{t:A}
Suppose that $ A : \cH \to \cH $ is a bounded operator with the
property  that for
any $ \lambda \in \RR $ there exists a constant $ C ( \lambda ) $ such 
that for $ \varphi \in \cH_P^2 $
\begin{equation}
\label{eq:AA}     \| \varphi \|_{ \cH} \leq C( \lambda ) \left( \| ( P - \lambda ) \varphi \|_{
  \cH} + \| A \varphi \|_{ \cH } \right). 
\end{equation}
Suppose also that for some $ \epsilon > 0 $, $ T > 0$, $ C_1$ and $
C_2 $, 
\begin{equation}
\label{eq:AA1}     \| \varphi \|_{ \cH}^2 \leq C_1 \int_0^t \| A U ( s ) \varphi \|_{ \cH }^2
ds + C_2 \| \varphi \|_{\cH^{-\epsilon}_P }^2 ,  \ \ \frac T 4  \leq t
\leq T . 
\end{equation}
Then there exist explicitely computable constant $ K $
such that
\begin{equation}
\label{eq:AA2}     \| \varphi \|_{ \cH}^2 \leq K \int_0^T \| A U ( t ) \varphi \|_{ \cH }^2
dt  .
\end{equation}
\end{thm}

\medskip
\noindent
{\bf Remarks.} 1. We do not compute the constant explicitely 
but the construction in the proof certainly allows that.

2. In the applications in this paper 
\[ P = - \Delta + V, \ \ \ \cH = L^2 ( \T^2) , \ \ \  A =
\bbbone_\Omega , \ \ \ \Omega \subset \T^2 \ \text{open}, \]
 or 
\[  P = - ( \partial_x + i k)^2 +
W , \ \  \cH = L^2 ( \T^1 ) , \ \  A = \bbbone_\omega,
\ \ \ \omega \subset \T^1 \ \text{open}, \]

\begin{proof}
We start by observing that \eqref{eq:AA1} and the definition
\eqref{eq:defHP} imply that for $ N > (2 C_2 )^{1 /{
    \epsilon } } $,
\begin{gather}
\label{eq:A1}
\begin{gathered}
   \| ( I - \Pi)  \varphi \|^2 \leq 2 C_1 \int_0^t \| A U ( s ) ( I - \Pi) \varphi \|^2
ds   ,  \ \ \ 
\frac T 4  \leq t \leq T , \\
\Pi \varphi := \sum_{ \lambda_n \leq N} \langle \varphi, \varphi_n \rangle  \varphi_n .
\end{gathered}
 \end{gather}
For reasons which will be explained below we will use this inequality
for $ t = T/4 $ and apply it $ \varphi $ replaced by $  U ( T/2 )
\varphi $: 
\begin{equation}
\label{eq:A2} 
  \| ( I - \Pi ) \varphi \|^2 \leq 2 C_1 \int_{T/2}^{3T/4} \| A U ( t )
( I - \Pi )   \varphi \|^2 dt . \end{equation}

We will show that the same estimate is true for $ \Pi \varphi $. 
For that let $ \mu_1 < \mu_2 <  \cdots < \mu_{r_1}  $ be the enumeration of  $ \{
\lambda_n \}_{n=1}^{K_1}  $ and define
\[ \psi_r := \sum_{\lambda=n = \mu_r } \langle \varphi , \varphi_n
\rangle  \varphi_n , \]
so that
\[ U ( t ) \Pi \varphi = \sum_{ n \leq K_1 } e^{ - i \lambda_n t } \langle \varphi , \varphi_n
\rangle  \varphi_n  = \sum_{ r=1}^{r_1} e^{i \mu_r t } \psi_r .\]
Since $ ( P - \mu_r ) \psi_r = 0 $,  we can apply \eqref{eq:AA}
to obtain 
\begin{equation}
\label{eq:psir}  \| \psi_r \|  \leq K_2 \| A \psi_r \| , \ \ K_2 = \max_{ n
  \leq K_1 } C ( \lambda_n ) . \end{equation}

The functions $ t \mapsto e^{i \mu_r
  t  } $, $ r = 1, \cdots, r_1 $, are linearly independent there exists a constant 
\[ K_3 = K_3 ( \mu_1, \cdots, \mu_{r_1} , T ) \]
such that for any $ f_1, \cdots , f_{r_1} \in \cH  $, 
\begin{equation}
\label{eq:A21} 
\int_{T/2}^{3T/4} \| \sum_{ r =1}^{r_1} e^{ i \mu_r t } f_r  \| ^2 dt \geq
K_3  \sum_{ r=1}^{r_1}  \| f_r \| ^2 , 
\end{equation} 
as both sides provide equivalent norms on $ \times_{ r=1}^{r_1} \cH
$. 

Applying \eqref{eq:A21} with $ f_r = A \psi_r $ and \eqref{eq:psir} gives
\begin{equation}
\label{eq:36}
\begin{split} 
\|A U ( t ) \Pi \varphi \|_{ L^2 ( (T/2, 3T/4 ); \cH) }^2 & = 
\int_{T/2}^{3T/4} \| \sum_{ r=1}^{r_1} A \psi_r e^{ i \mu_r t}
\|^2 dt \geq K_2  \sum_{ r=1}^{r_1}  \| A \psi_r \|^2 
\\
& \geq K_2 K_3 \sum_{ r=1}^{r_1} \| \psi_r \|^2 = K_2 K_3 \| \Pi \varphi \| . 
\end{split}
\end{equation}

The combination of \eqref{eq:A2} and \eqref{eq:36} do not yet provide 
the estimate \eqref{eq:AA2}. However if 
\[ \Pi_M \varphi := \sum_{\lambda_n \leq M } \langle \varphi , \varphi_n
\rangle \varphi_n , \]
then, for $ M $ sufficiently large we have
\begin{equation}
\label{eq:38}
\begin{split}
& \| A  U ( t ) ( I - \Pi_M + \Pi ) \varphi \|_{ L^2 ( [0 , T ] ; \cH )}
^2 
\geq \\ 
& \ \ \ \ \ \ \ \ \ K_2^2 K_3^2 \| \Pi \varphi \| ^2 + ( 1/ 4 C_1^2 ) \| (I -\Pi_M )
\varphi \| ^2 - K_4 M^{-1} \| \varphi\|^2 .
\end{split}
\end{equation}
where $ K_4 $ will be defined below. In fact, we choose
$ \eta \in \CIc ( ( 0 , T ) ) $ equal to $ 1 $ on $ [T/2, 3T/4]$, 
then the left hand side in \eqref{eq:38} is estimated
from below by 
\[ \begin{split}
\int \| A U ( t ) ( I - \Pi_M + \Pi ) \varphi \| ^2 \eta ( t ) dt &
=  \int \| A U ( t ) ( I - \Pi_M ) \varphi \| ^2 \eta ( t ) dt 
+ 
\int \| A U ( t ) \Pi \varphi \| ^2 \eta ( t ) dt \\ & \ \ \ \ - \,
2 \Re \int \langle A U ( t ) ( I - \Pi_M ) \varphi, A U ( t ) \Pi \varphi
\rangle \eta ( t ) dt .\end{split}\]
We can apply \eqref{eq:A1} and \eqref{eq:36} to estimate the
first to terms from below. Since
\[ \begin{split}
& 2 \Re \int \langle A U ( t ) ( I - \Pi_M ) \varphi, A U ( t ) \Pi \varphi
\rangle  \eta ( t )  dt = \\
&  \ \ \ 2 \Re \sum_{ \lambda_n < N } \sum_{
  \lambda_m > M } \langle \varphi , \varphi_n \rangle  
\langle \varphi_m, \varphi \rangle \langle A \varphi_n,
A \varphi_m \rangle \int e^{ i ( \lambda_n - \lambda_m ) t }
\eta ( t ) dt \\
& \leq C_P \| A\|^2 \sum_{ \lambda_n < N } \sum_{
  \lambda_m > M }  | \lambda_n - \lambda_m |^{-P} \| \varphi\|^2 \leq 
K_4 M^{-1}\| \varphi\|^2 , 
\end{split}
\]
if we choose $ P $ sufficiently large. This proves \eqref{eq:38}

We now have to deal with the remaining eigenfuctions corresponding to 
$ N \leq \lambda_n < M $.  For that let  $ \mu_{r_1 + 1}  < \cdots <
\mu_{r_2 } $ be the enumeration of these eigenvalues.  Put
\begin{equation}
\label{eq:tau}   \tau = \frac T { 10 r_2 } .\end{equation}
The Vandermonde matrix $ ( e^{ i \mu_r  p \tau  } )_{ 1 \leq r \leq r_2, 1 \leq
  p \leq r_2 } $ is non-singular and hence we can find scalars $
\sigma_p $,  $ \max | \sigma_p | = 1$, satisfying
\begin{equation}
\label{eq:39}  \begin{split}
& \sum_{p=1}^{r_2} \sigma_p e^{ i \mu_r p \tau} = 0  \ \ \text{ for }
 \ r \leq r_1 , \ \ \
 | \sum_{p=1}^{r_2} \sigma_p e^{ i \mu_r p \tau} | \geq K_5 \ \ \text{
  for }  r_1 < r \leq r_2 , 
\end{split}
\end{equation}
with a constant $ K_5 = K_5 ( \mu_1, \cdots, \mu_{r_2}, T  )
$. (Note the implicit dependence on $ M$.)

If we define 
\begin{equation}
\label{eq:deftp}  \tilde \varphi = \sum_{ \lambda_n > N } \left( \sum_{r=1}^{r_2}
  \sigma_p e^{ i \lambda_n p \tau } \right) \langle \varphi, \varphi_n
\rangle \varphi_n , \end{equation}
then 
\begin{equation}
\label{eq:pip}  ( I - \Pi)  \tilde \varphi = \tilde \varphi ,  \ \
\text{and} \ \ 
U ( t ) \tilde \varphi  = \sum_{ r = 1}^{r_2 } \sigma_p U ( t + p \tau
) \varphi . 
\end{equation}
Applying \eqref{eq:A1}, \eqref{eq:39} and the definition
\eqref{eq:deftp} gives
\[  \begin{split}4 C_1^2 \| A U ( t ) \widetilde \varphi \|_{L^2 ( [T/2, 3T/4]; \cH) }^2
& \geq  \| \tilde \varphi \|^2 \geq 
 \sum_{ N \leq \lambda_n < M } \left| \sum_{r=1}^{r_2} \sigma_p e^{ i
        \lambda_n p \tau} \right|^2  | \langle \varphi, \varphi_n
    \rangle|^2  \\
& \geq K_5^2 \| ( \Pi_M - \Pi ) \varphi \|^2 . 
\end{split}
\]
The choice of  $ \tau $ in \eqref{eq:tau} and \eqref{eq:pip} show that
\begin{equation}
\label{eq:pip2}
\|A  U ( t ) \varphi \| \geq \frac{ K_5} { 2 C_1 r_2 } \| ( \Pi_M -
\Pi ) \varphi \|^2 .
\end{equation}
This gives,
\[ \begin{split} \| A  U ( t ) ( I - \Pi_M + \Pi ) \varphi \|_{ L^2 ( [0 , T ] ;
  \cH)} & \leq \|A U ( t )\varphi\|_{ L^2 ( [0 , T ] ; \cH)}
+ \sqrt T \| ( \Pi_M - \Pi ) \varphi \|  \\
& \leq  \left( 1 + \frac{ 2 \sqrt T r_2 C_1 }{ K_5} \right) \|A  U ( t ) \varphi
\|_{ L^2 ( [ 0 , T ] , \cH ) }, \end{split} \]
which combined with \eqref{eq:38} and \eqref{eq:pip2} produces
\[  \begin{split} \left( 1 + \frac{ 2 ( \sqrt T + 1 ) r_2 C_1 }{ K_5} \right) \|A  U ( t ) \varphi
\|_{ L^2 ( [ 0 , T ] , \cH ) } & \geq  K_2 K_3 \| \Pi \varphi \| 
+ 1/ (2 C_1) \| (I -\Pi_M ) \varphi \| \\ 
& \ \ \ + \, \| ( \Pi_M - \Pi) \varphi \| - \sqrt { K_4/ M} \| \varphi\|^2 \\
& \geq ( K_6  - \sqrt { K_4 /  M } ) \| \varphi \|.
\end{split} \]
Since $ K_6 $ and  $ K_4 $ are independent of $ M $ we obtain 
\eqref{eq:AA2} by choosing $ M $ large enough.
\end{proof}
\section{ Proof of Lemma~\ref{geom}}\label{app.B}
This is a purely geometric result  which does not involves integer
points. It is the consequence of the fact that the circle is curved
but we prove it by explicit calculations. 

We start with the case where $\gamma =1$ (recall that in Lemma~\ref{geom} the modulus is defined by $|(x_1, x_2)|^2 = x_1^2 + \gamma x_2 ^2$).  
We perform a change of variables $x\mapsto  x h $, and denote by
$\epsilon = \kappa^2 h^2 $.
We are reduced to proving that for 
\begin{equation}
\begin{gathered}
   \mathcal{B}_{\epsilon,\alpha}= \{ z\in \CC \; : \; \Re z \geq 0, \ \Im z
   \geq 0, \  | |z| - 1 | \leq \epsilon, \ \arg (z) \in [\alpha \sqrt{\epsilon}, (\alpha+1)\sqrt{\epsilon}) \}.
\end{gathered}
\end{equation} 
we have 
\begin{lem} There exists $\epsilon _0 >0$ and $Q>0$ such that for any $0< \epsilon\leq \epsilon _0$,  we have 
\begin{equation}
\begin{gathered}
\forall \alpha_j \in \{ 0, 1, \dots, N_{\epsilon}\}, j=1, \dots 4, \ \
N_\epsilon := \left[ \frac{ \pi } { 2 \sqrt \epsilon } \right] \\
 (\mathcal{B}_{\epsilon,\alpha_1}+\mathcal{B}_{\epsilon,\alpha_2})\cap (\mathcal{B}_{\epsilon,\alpha_3}+\mathcal{B}_{\epsilon,\alpha_4})\neq \emptyset \\
\Longrightarrow \ | \alpha_1 - \alpha_3| + | \alpha_2- \alpha_4 |\leq Q \ 
\ \text{ or } \ \ | \alpha_1 - \alpha_4| + | \alpha_2- \alpha_3 | \leq Q 
\end{gathered}
\end{equation}
\end{lem}

\begin{proof}
We first observe that it is enough to prove 
the lemma with the condition $ | | z | -1 | < \epsilon $ 
replaced by $ 0 \leq | z | -1 \leq \epsilon $ in
the definition of $ {\mathcal B}_{\epsilon, \alpha} $:
$  0 \leq 1 - |z| \leq \epsilon $ is the same as
$ 0 \leq |z|/(1-\epsilon) - 1 \leq \epsilon/( 1 - \epsilon) $.

Let $z_j = \rho_j e^{i\theta_j}\in \mathcal{B}_{\epsilon, \alpha_j}$,
$ 1 \leq j \leq 4$, be such that $z_1+ z_2 = z_3 + z_4$. 
By possibly exchanging $z_1$ and $z_2$ 
we can assume $\theta_1 \geq \theta_2$ and similarly that 
$ \theta_3 \geq \theta_4$. In particular,
\begin{equation}
\label{eq:B.new}
\textstyle{ \frac{\theta_1- \theta_2 } 2 } \in [0 ,  \textstyle{ \frac \pi 4} ], 
\ \   \frac{\theta_3- \theta_4} 2 \in [0 , \textstyle{ \frac \pi 4}]. 
\end{equation}

Since $\rho_j \in [1, 1+ \epsilon]$, we have 
\[  | e^{i\theta_1} + e^{i\theta_2} - e^{i\theta_3} - e^{i\theta_4} | \leq 4 \epsilon , \]
which is the same as 
\begin{equation}\label{eq.3}|e^{\frac i 2( { \theta_1 + \theta_2} )} \cos ( 
\textstyle{\frac{ \theta_1 - \theta_2} 2} )- e^{\frac i 2 ({ \theta_3 + \theta_4} )} 
\cos ( \textstyle{\frac{ \theta_3 - \theta_4} 2})| \leq 
2\epsilon
\end{equation}
On the other hand,
\[ \begin{split}
|e^{ \frac i 2( { \theta_1 + \theta_2} )} 
\cos ( \textstyle{\frac{ \theta_1 - \theta_2} 2})- 
e^{ \frac i 2 ( { \theta_3 + \theta_4} )} \cos ( 
\textstyle{\frac{ \theta_3 - \theta_4} 2} )| & =  |e^{\frac i 2  
( { \theta_1 + \theta_2 - \theta_3 - \theta_4} )} \cos ( 
\textstyle{\frac{ \theta_1 - \theta_2} 2})- \cos ( 
\textstyle{\frac{ \theta_3 - \theta_4} 2} )|\\
& \geq| \sin( \textstyle{\frac{ \theta_1 + \theta_2 - \theta_3 - \theta_4 } 2 })\cos ( \textstyle{\frac{ \theta_1 - \theta_2} 2}) |  
\end{split} .\]
Since \eqref{eq:B.new} implies that $ \cos (
{\frac{ \theta_1 - \theta_2} 2}) \geq 1/\sqrt 2 $, we obtain 
from \eqref{eq.3} that
$$ |\sin( \textstyle{\frac{ \theta_1 + \theta_2 - \theta_3 - \theta_4 } 2} )| \leq 2\sqrt{ 2} \epsilon.$$
We also have 
$  \frac{ \theta_1 + \theta_2 - \theta_3 - \theta_4 } 2 \in [ - \frac \pi 2, \frac \pi 2]  $ 
and as $ |\sin \theta | \geq 2 |\theta|/\pi $ for 
$ -\pi/2 \leq \theta \leq \pi/2 $, 
we conclude that
\begin{equation}
\label{eq:B.new1}  \left|\textstyle{\frac{ \theta_1 + \theta_2 - \theta_3 - \theta_4 } 2} \right| \leq  \pi \sqrt{2} \epsilon  .
\end{equation}
We assumed that $ z_j = \rho_j e^{ i \theta_j } \in 
{\mathcal B }_{\epsilon,\alpha_j } $ and that means that
$ 0 \leq \theta_j - \sqrt{ \epsilon} \alpha_j < \sqrt{ \epsilon}$.
Hence \eqref{eq:B.new1} gives
 \begin{equation}\label{estimee1}
 \left| \alpha_1 + \alpha_2 - \alpha_3 - \alpha_4 \right| \leq C \sqrt{ \epsilon} + 2 \leq 3,   
 \end{equation}
provided that  $\epsilon >0$ small enough.

 Going back to \eqref{eq:B.new} and \eqref{eq.3}
we get with $p = \frac{ \theta_1- \theta_2} 2$, $q=\frac{ \theta_3- \theta_4} 2$
 \begin{equation} | \cos p 
 - \cos q | = 2 | \sin\left( 
\textstyle{\frac {p+q} 2 }\right) \sin \left( 
\textstyle{\frac {p-q} 2} \right)| \leq 2 \epsilon
 \end{equation}
As, $p,q \in [0, \frac \pi 4]$ we get 
 $$ | \textstyle{\frac{ (p+q)} 2  \frac {(p-q) } 2} | \leq  
\textstyle{\frac{ \pi ^2} 4 } \epsilon.
 $$
This is the same as (recall that $ 0 \leq \theta_1 - \theta_2 $, 
$ 0 \leq \theta_3 - \theta_4 $)
 \begin{equation} \label{eq.4}
   ( |\theta_1 - \theta_2| - |\theta_3- \theta_4| ) ( |\theta_1 - \theta_2| +| \theta_3 - \theta_4| )   \leq 4 \pi^2 \epsilon
  \end{equation}
and this gives
 \begin{equation}\label{eq.5} 
 |  (\theta_1 - \theta_2) - ( \theta_3 - \theta_4 ) | \leq 
\left(  ( |\theta_1 - \theta_2| - |\theta_3- \theta_4| ) ( |\theta_1 - \theta_2| +| \theta_3 - \theta_4| )  \right)^{\frac12} 
  \leq 2 \pi \sqrt{ \epsilon}
 \end{equation}
Using again the fact that 
$ 0 \leq \theta_j - \sqrt{ \epsilon} \alpha_j < \sqrt{ \epsilon}$
this gives
 \begin{equation}\label{eq.6}
|  (\alpha_1- \alpha_2 ) - ( \alpha_3- \alpha_4)| \leq 2 \pi + 2 
 \end{equation}
 Finally, from~\eqref{estimee1} and ~\eqref{eq.6} we obtain
 $$ | \alpha_1 - \alpha_3|\leq \pi + \textstyle{\frac 5 2} ,    \ \ \ 
| \alpha_2 - \alpha_4|\leq \pi + \textstyle{\frac 5 2} $$
 which proves Lemma~\ref{geom}  in the case $\gamma =1$ (notice that here only the first term in the alternative is possible which follows from the assumption $\theta_1 \geq \theta_2, \theta_3 \geq \theta_4$). The general case follows by applying the transformation 
$ (x_1, x_2)\in \R^2 \mapsto (x_1, \sqrt{ \gamma} x_2) \in \R^2.$
  \end{proof}


\begin{thebibliography}{10}

\bibitem{AM}
N. Anantharaman and F. Macia, 
\newblock
Semiclassical measures for the Schr\"odinger equation on the torus,
\newblock
{\tt arXiv:1005.0296}.


\bibitem{BLR} 
C. Bardos, G. Lebeau and J. Rauch.
\newblock 
Sharp sufficient conditions for the observation, control, and stabilization of waves from the boundary.
\newblock {\em SIAM J. Control Optim.} 30:1024--1065, 1992.
\bibitem{BSSY}
J. Bourgain, P. Shao, C.D. Sogge and X. Yao
\newblock{On $L^p$-resolvent estimates and the density of eigenvalues for compact Riemannian manifolds}
\newblock{ preprint arXiv:1204.3927,} 2012

\bibitem{B} 
N.~Burq.
\newblock Semi-classical measures for inhomogeneous Schr\"odinger
equations on tori,
\newblock{\tt arXiv:1209.3739}, to appear in {\em Analysis \& PDE}.
\bibitem{BuGeTz} 
{N. Burq, P. G\'erard, and N. Tzvetkov},
\newblock{An instability property of the nonlinear Schr\"odinger equation on $S^d$}.
\newblock{\em Math. Res. Lett.}, 9(2-3):323--335, 2002.


\bibitem{BZ2}
N.~Burq and M. Zworski.
\newblock Geometric control in the presence of a black box.
\newblock {\em Jour. A.M.S.} 17, 2004, no. 2, 443--471.


\bibitem{BZ3}
N.~Burq and M. Zworski.
\newblock Bouncing ball modes and quantum chaos.
\newblock {\em SIAM Review}, {43--49}, 47, 2005.


\bibitem{BZ4}
N.~Burq and M. Zworski.
\newblock Control for Schr\"odinger equations on tori
\newblock {\em Math. Research Letters} 19: 309-324, 2012.



\bibitem{Co} 
A. C\'ordoba, 
\newblock Geometric Fourier analysis, 
\newblock {Annales de l'institut Fourier}, 32:215--226, 1982.


\bibitem{DiSj}
M. Dimassi and J. Sjšstrand
\newblock Spectral asymptotics in the semiclassical limit.
{\it  London Mathematical Society Lecture Note Series}, 268. Cambridge University Press, Cambridge, 1999. xii+227 pp. 

\bibitem{KDSS}
D. Dos Santos, C. Kenig and M. Salo
\newblock{On Lp resolvent estimates for Laplace-Beltrami operators on compact manifolds}
\newblock{preprint, arXiv:1112.3216, 31 pages, to appear in Forum mathematicum,} 2011.


\bibitem{Ha}
A. Haraux.
\newblock S\'eries lacunaires et contr\^ole semi-interne des vibrations
              d'une plaque rectangulaire,
\newblock{\em J. Math. Pures Appl.} 68-4:457--465, 1989.

\bibitem{Ja} S. Jaffard.
\newblock Contr\^ole interne exact des vibrations d'une plaque rectangulaire. 
\newblock {\em Portugal. Math.} 47 (1990), no. 4, 423-429.


\bibitem{JK} 
 D. Jerison and C.E. Kenig
\newblock Unique continuation and absence of positive eigenvalues for
Schr\"odinger operators, 
\newblock {\em Ann. Math.} 121(3),  463-488, 1985.


\bibitem{Ka} 
J.P. Kahane.
\newblock
Pseudo-p\'eriodicit\'e et s\'eries de Fourier lacunaires
\newblock 
{\em Ann. Sci.  \'Ecole Norm. Sup.} 79 (1962), no.3,  93--150.



\bibitem{Le}G. Lebeau
\newblock Contr\^ole de l'\'equation de Schr\"odinger
\newblock{\em J. Math. Pures Appl.} (9) 71, no. 3, 267--291, 1992.

\bibitem{Li}
J.L. Lions.
\newblock {\em Contr{\^o}labilit{\'e} exacte. Perturbation et stabilisation des
  syst{\`e}mes distribu{\'e}s}, volume~23 of {\em R.M.A.}
\newblock Masson, 1988.


\bibitem{Mi}
L. Miller
\newblock{Controllability cost of conservative systems: resolvent condition and transmutation}
\newblock{ \em J.  Funct. Anal.} 218, 2, 425-444, 2005.

\bibitem{SS}
M. Schechter and B. Simon, 
\newblock Unique continuation for Schr\"odinger operators with
unbounded potential,
\newblock {\em J. Math. Anal. Appl.} 77, 482-492, 1980.


\bibitem{EZB}
M. Zworski.
\newblock {\em Semiclassical analysis}, 
\newblock {\bf 138} {\em Graduate Studies in Mathematics}, AMS 2012.

\end{thebibliography}
\end{document}